\documentclass[11pt]{amsart}

\usepackage[a4paper, top=3.5cm, left=3cm, right=3cm, bottom=3.5cm]{geometry}

\usepackage{amsmath,amssymb,amsthm,dsfont,mathrsfs,stmaryrd}
\usepackage[T1]{fontenc}
\usepackage[utf8]{inputenc}

\usepackage{etoolbox}

\makeatletter
\patchcmd{\ttlh@hang}{\parindent\z@}{\parindent\z@\leavevmode}{}{}
\patchcmd{\ttlh@hang}{\noindent}{}{}{}
\makeatother

\usepackage[all,cmtip]{xy}
\usepackage[fleqn,tbtags]{mathtools}
\usepackage{mathabx}
\usepackage{url}

\usepackage{tikz,tikz-cd}
\usepackage{enumerate}
\usetikzlibrary{calc,fit}

\usepackage[numbers]{natbib}

\usepackage{hyperref}

\theoremstyle{plain}
\newtheorem{theo}{Theorem}
\newtheorem{prop}[theo]{Proposition}
\newtheorem{cor}[theo]{Corollary}

\newtheorem{lem}[theo]{Lemma}

\theoremstyle{definition}
\newtheorem{defi}[theo]{Definition}
\newtheorem{ex}[theo]{Example}

\theoremstyle{remark}
\newtheorem{rem}[theo]{Remark}

\DeclareMathOperator{\Seg}{Seg}
\DeclareMathOperator{\Mult}{Mult}

\DeclareMathOperator{\Irr}{Irr}
\DeclareMathOperator{\id}{Id}
\DeclareMathOperator{\soc}{soc}
\DeclareMathOperator{\cs}{cos}
\DeclareMathOperator{\sgn}{sgn}

\newcommand{\BC}{\mathbb C}

\newcommand{\BZ}{\mathbb Z}

\newcommand{\CC}{\mathcal C}

\newcommand{\CL}{\mathcal L}

\newcommand{\CR}{\mathcal R}

\newcommand{\LC}{{\rm LC}}
\newcommand{\RC}{{\rm RC}}
\newcommand{\LI}{{\rm LI}}
\newcommand{\RI}{{\rm RI}}
\newcommand{\GL}{{\rm GL}}

\newcommand{\m}{\mathfrak{m}}

\numberwithin{equation}{section}
\numberwithin{theo}{section}

\title{On a determinant formula for some real regular representations}
\author{L\'{e}a Bittmann}
\address{L\'{e}a Bittmann, Hodge Institute, University of Edinburgh, United Kingdom.}
\email{lea.bittmann@ed.ac.uk}
\date{} 

\begin{document}

\begin{abstract}
We interpret a formula established by Lapid-M\'{\i}nguez on real regular representations of $\GL_n$ over a local non-archimedean field as a matrix determinant. We use the Lewis Carroll determinant identity to prove new relations between real regular representations. Through quantum affine Schur-Weyl duality, these relations generalize Mukhin-Young's \emph{Extended $T$-systems}, for representations of the quantum affine algebra $U_q(\widehat{\mathfrak{sl}}_k)$, which are themselves generalizations of the celebrated \emph{$T$-system} relations. 
\end{abstract}

\maketitle

\setcounter{tocdepth}{1}
\tableofcontents

\section{Introduction}

The context of this work is the representation theory of $\GL_n(F)$ (where $F$ is a non-archimedean local field), or equivalently of the type $A$ quantum affine algebra $U_q(\widehat{\mathfrak{sl}}_k)$ (where $q\in \BC^\times$ is not a root of unity). Indeed, through Chari-Pressley's \emph{quantum affine Schur-Weyl duality} \cite{CP95}, the category of complex smooth finite-length representations of $\GL_n(F)$ is equivalent to the category of (level $n$) finite-dimensional $U_q(\widehat{\mathfrak{sl}}_k)$-modules, when $k\geq n$. Since both contexts are equivalent, we will work with the category $\CC$ of $\GL_n(F)$ representations in most of this paper. Both these categories have been intensively (and independently) studied, but some important natural questions remain open. 

The normalized parabolic induction, denoted by $\times$, endows this category with a ring category structure, and its Grothendieck group $\CR$ with a ring structure.
Irreducible representations in the category $\CC$ have been classified by Zelevinsky \cite{Z80} using \emph{multisegments} (formal sums of segments). For $\m = \Delta_1 + \Delta_2 + \cdots + \Delta_N$ a multisegment, the corresponding irreducible representation $Z(\m)$ is obtained as the unique irreducible subrepresentation of the \emph{standard representation} $\zeta(\m) := Z(\Delta_1)\times Z(\Delta_2)\times \cdots \times Z(\Delta_N)$.
The classes of the irreducible representations and the standard representations form two bases of the Grothendieck ring $\CR$, the change of basis matrix between them is unitriangular, with coefficients which can be expressed in terms of Kazhdan-Lusztig polynomials (see \cite{Z81}, \cite{CG97}). A similar story was established for finite-dimensional representations of $U_q(\widehat{\mathfrak{sl}}_k)$ (see the work of Nakajima \cite{QVFDR}). This gives an algorithm to compute the classes of the simple representations from the classes of the standard representations. However, in practice the actual computation of the coefficients can be very difficult. 

For some specific classes of irreducible representations, remarkable formulas have been established to compute their classes as linear combinations of classes of standard representations. The work of Tadić \cite{T95}, and then Chenevier-Renard \cite{CR08}, established such a formula for \emph{Speh representations}.
Cleverly, this formula can be seen as the computation of the determinant of a matrix, and it was then proved using the \emph{Lewis Carroll identity} (also called \emph{Dodgson’s rule of determinant}). In \cite{LM14}, Lapid-M\'{\i}nguez generalized Tadić's formula to a larger class of representations called \emph{ladder representations}. 
Then, in \cite{LM18} the same authors established an even more general formula (see \eqref{eq_alt_sum} below), for \emph{regular representations} which are \emph{real} - $Z(\m)$ such that $Z(\m)\times Z(\m)$ is irreducible. 

Furthermore, in \cite{LM14}, Lapid-M\'{\i}nguez used the Lewis Carroll identity to obtain a remarkable relation between the classes of some of these ladder representations. For $Z(\m)= Z(\Delta_1 + \cdots + \Delta_N$) a ladder representation, we have the following relation in $\CR$ \cite[Corollary 12]{LM14}:
\begin{multline}\label{eq_ext_T-sys_intro}
Z(\Delta_1 + \cdots + \Delta_{N-1})\times Z(\Delta_2+ \cdots + \Delta_{N}) = Z(\m)\times Z(\Delta_2 + \cdots + \Delta_{N-1}) \\
+ Z(\m')\times Z(\m''),
\end{multline}
where $Z(\m'),Z(\m'')$ are also ladders (see Theorem~\ref{theo_ladder}). 
Through quantum affine Schur-Weyl duality, relation \eqref{eq_ext_T-sys_intro} has been established independently by Mukhin-Young in \cite[Theorem~4.1]{MY12} for representations of the quantum affine algebra $U_q(\widehat{\mathfrak{sl}}_k)$, under the name \emph{Extended $T$-systems}.

The extended $T$-systems are generalizations of the famous \emph{$T$-system relations}, which are sets of recurrence relations of crucial importance in the study of certain integrable systems (see review \cite{KNS11}). For representations of quantum affine algebras, the $T$-systems are relations in the Grothendieck ring $\CR$ between classes of Speh representations (called \emph{Kirillov-Reshetikhin modules} there). These relations were proved in all simply-laced types ($A$, $D$ or $E$) by Nakajima \cite{QVFDR} and in all types by Hernandez \cite{KRC}. Additionally, the $T$-systems, and their extended version, can be interpreted as \emph{short exact sequences} between irreducible finite-dimensional $U_q(\widehat{\mathfrak{g}})$-modules.

More recently, the $T$-systems gained a new interpretation as \emph{exchange relations} in a Fomin-Zelevinsky cluster algebra \cite{CA1}. Indeed, in \cite{ACAA} Hernandez-Leclerc proved this interpretation of $T$-systems as cluster transformations and used it to the prove that the Grothendieck ring of the category of finite-dimensional $U_q(\widehat{\mathfrak{g}})$-modules (in all Dynkin types) had the structure of a cluster algebra. Note that Duan-Li-Luo obtained in \cite{DLL20} another generalization of the $T$-systems, different from Mukhin-Young extended $T$-systems, which they also interpreted as exchange relations in the cluster algebra structure.

In the present work, we establish formulas generalizing the extended $T$-systems of Mukhin-Young, for some real regular representations. Regular representations have a permutation associated to them and in \cite{LM18}, Lapid-M\'{\i}nguez gave a sufficient condition for a regular representation to be real, as a pattern avoidance condition on the permutation associated to the representation. We show, using the notion of \emph{Ferres boards} and the work of Sjostrand \cite{S07} that under the same pattern avoidance condition, Lapid-M\'{\i}nguez's formula \cite[Theorem 1.2 (9)]{LM18} can be written as a matrix determinant. Our relations are then obtained using some choice of Lewis Carroll identities. As our main result, we prove the following (Theorem~\ref{theo_main} and Corollary~\ref{cor_main}): for $Z(\m)= Z(\Delta_1 + \cdots + \Delta_N$) a regular representation such that the associated permutation $\sigma$ avoids the patterns $3412$ and $4231$, we have the following relations in $\CR$ (\eqref{eq_moins_N} and \eqref{eq_moins_1}):
\begin{align}
Z(\m\setminus \Delta_N)\times Z(\m\setminus \Delta_{\sigma(N)}) & = Z(\m)\times Z(\m\setminus \Delta_N,\Delta_{\sigma(N)}) + Z(\m_1')\times Z(\m_1''),\label{eq_main_intro_1}\\
Z(\m\setminus \Delta_1)\times Z(\m\setminus \Delta_{\sigma(1)}) & = Z(\m)\times Z(\m\setminus \Delta_1,\Delta_{\sigma(1)}) + Z(\m_2')\times Z(\m_2''),\label{eq_main_intro_2}
\end{align}
where $m_1',m_1'',m_2'$ and $m_2''$ are real regular representations.

As part of Theorem~\ref{theo_main} and Corollary~\ref{cor_main}, we also prove that, as the extended $T$-systems, these relations correspond to a decomposition of a module of length 2, i.e. the two terms in the right hand side of \eqref{eq_main_intro_1} and \eqref{eq_main_intro_2} are irreducible representations. We prove this using Lapid-M\'{\i}nguez's \cite{LM16} combinatorial irreducibility criteria, as well as a newly introduced notion of \emph{good segments} in a mutlisegment, which enables us to prove by induction that some parabolic induction of irreducible representations are irreducible.

The paper is organized as follows. We start with some reminders about segments, multisegments, $p$-adic representations of $\GL_n(F)$ and the Zelevinsky classification in Section~\ref{sect_prelim}. We also recall 
Lapid-M\'{\i}nguez's \cite{LM16} irreducibility criteria for a parabolic induction of two representations, using socles and cosocles.
In Section~\ref{sect_good}, we introduce the notion of good segments and use it to obtain some combinatorial criteria to prove that certain parabolic inductions $Z(\Delta) \times Z(\m)$, where $Z(\m)$ is a regular representation are irreducible. We also prove an existence result for good segments (Proposition~\ref{prop_existence}). In particular, we obtain that every regular representation whose permutation avoids the patterns $3412$ and $4231$ has at least two good segments, from which we can recover that such representations are real. In Section~\ref{sect_det_formula}, we use the notion of Ferres boards and results from  Sjostrand \cite{S07} and Chepuri--Sherman-Bennett \cite{CSB21} to write existing relations as determinants of matrices. The main result is stated and then proved in Section~\ref{sect_main}, in which we also give examples. Finally, in Section~\ref{sect_QAA} we translate our results to the context of quantum affine algebra representations, and give some perspective, in particular in relation to cluster algebras.

\subsection*{Acknowledgements}

We would like to thank Alberto Mínguez for providing inspiration for this work. The author was partially supported by the European Research Council (ERC) under the European Union's Horizon 2020 research and innovation programme under grant agreement No 948885 and by the Royal Society University Research Fellowship.

\section{Preliminaries}\label{sect_prelim}

\subsection{Segments and multisegments}

\begin{defi}
A \emph{segment} is a pair of integers $a\leq b\in \BZ$, denoted by $[a;b]$. 

Let $\Seg$ denote the set of segments.
\end{defi}

The extremities of the segment $\Delta =[a;b]\in \Seg$ are denoted by $b(\Delta)=a$ and $e(\Delta)=b$.
We also write $\overleftarrow{\Delta} = [a-1;b-1]$.

\begin{defi}
Two segments $\Delta=[a;b]$ and $\Delta'=[c;d]$ are \emph{linked} if
\[
\begin{array}{rlll}
& a<c & \text{and} & c-1\leq b<d,\\
\text{or} & c<a & \text{and} & a-1\leq d <b.
\end{array}
\]
In the first case, we say that $\Delta $ \emph{precedes} $\Delta'$ and write $\Delta \prec \Delta'$.
\end{defi}

\begin{ex}
A few examples of linked and unlinked pairs of segments:
\begin{center}
\begin{tikzpicture}[scale=0.7]
\node[label={[yshift=-4pt]1}]  at (0,0) {};
\node[label={[yshift=-4pt]3}]  at (3,0) {};
\node[label={[yshift=-4pt]4}]  at (4,0) {};
\node[label={[yshift=-4pt]5}]  at (5,0) {};
\draw[very thick] (0,0) to (3,0);
\draw[very thick] (4,0) to (5,0);
\node[anchor=west] at (5.5,0.2) {are linked.};

\node[label={[yshift=-4pt]1}]  at (0,-1) {};
\node[label={[yshift=-4pt]2}]  at (2,-1) {};
\node[label={[yshift=-4pt]4}]  at (4,-1) {};
\node[label={[yshift=-4pt]5}]  at (5,-1) {};
\draw[very thick] (0,-1) to (2,-1);
\draw[very thick] (4,-1) to (5,-1);
\node[anchor=west] at (5.5,-0.8) {are not linked.};

\node[label={[yshift=-4pt]1}]  at (11,0) {};
\node[label={[below,yshift=-7pt]4}]  at (15,0) {};
\node[label={[yshift=-4pt]3}]  at (14,0.5) {};
\node[label={[yshift=-4pt]5}]  at (16,0.5) {};
\draw[very thick] (11,0) to (15,0);
\draw[very thick] (14,0.5) to (16,0.5);
\node[anchor=west] at (16.5,0.5) {are linked.};

\node[label={[yshift=-4pt]1}]  at (11,-1) {};
\node[label={[yshift=-4pt]5}]  at (16,-1) {};
\node[label={[below,yshift=-7pt]2}]  at (13,-1.5) {};
\node[label={[below,yshift=-7pt]4}]  at (15,-1.5) {};
\draw[very thick] (11,-1) to (16,-1);
\draw[very thick] (13,-1.5) to (15,-1.5);
\node[anchor=west] at (16.5,-1) {are not linked.};
\end{tikzpicture}
\end{center}
\end{ex}

\begin{defi}
A \emph{multisegment} $\m$ is a finite formal sum of segments of $\Seg$ (with possible multiplicities), $\m=\Delta_1 +\Delta_2 + \cdots + \Delta_N$.
Let $\Mult$ denote the set of multisegments.
\end{defi}

A sequence of segments $(\Delta_1,\ldots,\Delta_N)$ is said to be \emph{ordered} if, for all $1\leq i<j\leq N$, $\Delta_i$ does not precedes $\Delta_j$. If $\m\in\Mult$, and $(\Delta_1,\ldots,\Delta_N)$ is an ordered sequence of segments such that $\m = \Delta_1 +\cdots +\Delta_N$, we say that $(\Delta_1,\ldots,\Delta_N)$ is an \emph{ordered form} of $\m$.

\subsection{Representations}\label{sect_rep}

Let $F$ be a non-archimedean local field with a normalized
absolute value $|\cdot|$ and let $D$ be a finite-dimensional central division $F$-algebra. For $n\in \BZ_{\geq 1}$, let $\CC(\GL_n)$ be the category of complex, smooth representations of $\GL_n(D)$ of finite length and $\Irr(\GL_n)$ the set of equivalence classes
of irreducible objects of $\CC(\GL_n)$.
For $\pi_i \in \CC(\GL_{n_i})$, $i=1,2$, denote by $\pi_1 \times \pi_2 \in \CC(\GL_{n_1+n_2})$ the representation which is parabolically induced from $\pi_1 \otimes \pi_2$.
The parabolic induction endows the category $\bigoplus_{N \geq 0}\CC(\GL_n)$ with the structure of a tensor category.

For any supercuspidal representation $\rho \in \bigcup_{n\in \BZ_{\geq 0}}\Irr(\GL_n)$, there exists a unique positive real number $s_\rho$ such that $\rho|\cdot|^{s_{\rho}}\times \rho$ is reducible. Let $\nu_{\rho}= |\cdot|^{s_{\rho}}$, we write $\overrightarrow{\rho} = \rho \nu_{\rho}$, $\overleftarrow{\rho}=\rho \nu_{\rho}^{-1}$.
A \emph{cuspidal line} is an equivalence class on $\bigcup_{n\in \BZ_{\geq 0}}\Irr(\GL_n)$ for the equivalence relation given by $\rho \sim \overrightarrow{\rho}$.

For a fixed cuspidal line $\CL$, consider $\CC_\CL$ the Serre ring subcategory of $\bigoplus_{N \geq 0}\CC(\GL_n)$ consisting of the representations whose supercuspidal support is contained in $\CL$. Then all categories $\CC_\CL$ are equivalent as ring categories and the study of $\bigoplus_{N \geq 0}\CC(\GL_n)$ amounts to the study of one $\CC_\CL$. From now on, we fix a cuspidal line, drop the subscript and consider the category $\CC$, its set of equivalence classes of irreducible objects $\Irr$ and its Grothendieck ring $\CR$.

For $\Delta=[a;b]\in \Seg$, consider the induced representation
\[ I[a;b] := \rho\nu_\rho^a \times \rho\nu_\rho^{a+1} \times \cdots \times \rho\nu_\rho^b.
\]

\begin{defi} We consider the socle and cosocle of this representation:
\begin{align*}
Z[a;b] :=\soc(I[a;b]), & \quad \text{maximal semi-simple submodule,}\\
L[a;b] :=\cs(I[a;b]), & \quad \text{maximal semi-simple quotient.}
\end{align*}
\end{defi}

The following is known (see for example \citep{Z80}).

\begin{prop}\label{prop_notprime}
For $\Delta_1,\ldots,\Delta_N\in \Seg$, $Z(\Delta_1)\times \cdots \times Z(\Delta_N)$ (resp. $L(\Delta_1)\times \cdots \times L(\Delta_N)$) is irreducible if and only if the segments $\Delta_1,\ldots,\Delta_N$ are pairwise unlinked.
\end{prop}

For $\m\in\Mult$ and $(\Delta_1,\ldots,\Delta_N)$ an ordered form of $\m$, define the \emph{standard module}:
\[\zeta(\m):= Z(\Delta_1)\times \cdots \times Z(\Delta_N).\]
From the previous proposition, $\zeta(\m)$ does not depend on the chosen order.

\begin{theo}\cite{Z80}[Zelevinsky Classification]
The map
\begin{equation*}
\m  \mapsto Z(\m):=\soc(\zeta(\m)),
\end{equation*}
defines a bijection 
\[\Mult  \xrightarrow{\sim} \Irr.\]
\end{theo}

\subsection{Families of representations}

We are interested in some particular families of representations.
Let $Z(\m)$ be an irreducible representation, with $\m = \Delta_1 + \cdots + \Delta_N \in \Mult$.

\begin{defi}\label{def_speh}
The irreducible representation $Z(\m)$ is a \emph{Speh} representation if $\Delta_{i+1} = \overleftarrow{\Delta_i}$, for all $1\leq i\leq N-1$.
\end{defi}

\begin{ex}
The representations corresponding to the multisegments
 \[
\begin{tikzpicture}[scale=0.7]
\node[anchor=east, scale=0.8] at (-0.5,1) {$[3;3]+ [2;2]+ [1;1]+[0;0]=$};
\node[label={[yshift=-4pt]0}] at (0,0) {$\bullet$};
\node[label={[yshift=-4pt]1}] at (1,0.5) {$\bullet$};
\node[label={[yshift=-4pt]2}] at (2,1) {$\bullet$};
\node[label={[yshift=-4pt]3}] at (3,1.5) {$\bullet$};
\node at (4,1) {and};
\node[anchor=east,, scale=0.8] at (9.5,1) {$[2;4]+[1;3]+[0;2]=$};
\node[label={[yshift=-4pt]4}] at (14,1.5) {};
\node[label={[yshift=-4pt]2}] at (12,1.5) {};
\node[label={[yshift=-4pt]1}] at (11,1) {};
\node[label={[yshift=-4pt]0}] at (10,0.5) {};
\node[label={[below, yshift=-6pt]3}] at (13,1) {};
\node[label={[below, yshift=-6pt]2}] at (12,0.5) {};
\draw[ultra thick] (14,1.5) to (12,1.5);
\draw[ultra thick] (13,1) to (11,1);
\draw[ultra thick] (12,0.5) to (10,0.5);
\end{tikzpicture}
\]
are Speh representations.
\end{ex}

\begin{defi}\label{def_ladder}
The irreducible representation $Z(\m)$ is a \emph{ladder} representation if, for all $1\leq i\leq N-1$, $e(\Delta_{i+1}) < e(\Delta_i)$ and  $b(\Delta_{i+1}) < b(\Delta_i)$.
\end{defi}

\begin{ex}
All Speh representations are particular cases of ladder representations.

The representations corresponding to the multisegments
\[
\begin{tikzpicture}[scale=0.7]
\node[anchor=east, scale=0.8] at (-0.5,1) {$[2;5]+[1;3]+[0;0]=$};
\node[label={[yshift=-4pt]0}] at (0,0) {$\bullet$};
\node[label={[yshift=-4pt]1}] at (1,0.5) {};
\node[label={[yshift=-4pt]2}] at (2,1) {};
\node[label={[yshift=-4pt]5}] at (5,1) {};
\node[label={[below, yshift=-4pt]3}] at (3,0.5) {};
\draw[ultra thick] (1,0.5) to (3,0.5);
\draw[ultra thick] (2,1) to (5,1);

\node at (5.5,1) {,};

\node[anchor=east,, scale=0.8] at (9.5,1) {$[4;7]+[1;2]=$};
\node[label={[yshift=-4pt]4}] at (13,1.5) {};
\node[label={[yshift=-4pt]7}] at (16,1.5) {};
\node[label={[yshift=-4pt]2}] at (11,1) {};
\node[label={[yshift=-4pt]1}] at (10,1) {};

\draw[ultra thick] (13,1.5) to (16,1.5);
\draw[ultra thick] (11,1) to (10,1);
\end{tikzpicture}
\]
are ladder representations.
\end{ex}

\begin{defi}
The irreducible representation $Z(\m)$ is a \emph{regular} representation if, for all $1\leq i\neq j\leq N$, $e(\Delta_{j}) \neq e(\Delta_i)$ and $b(\Delta_{j}) \neq b(\Delta_i)$.
By extension, the multisegment $\m$ is also called \emph{regular}.
\end{defi}

\begin{ex}
All ladder representations are particular cases of regular representations.

The representation
\[
\begin{tikzpicture}[scale=0.8]
\node[anchor=east, scale=0.8] at (-0.5,1) {$Z([1;5]+[0;4]+[2;3])=$};
\node[label={[yshift=-4pt]1}] at (1,1) {};
\node[label={[yshift=-4pt]5}] at (5,1) {};
\node[label={[yshift=-4pt]0}] at (0,0.5) {};
\node[label={[below, yshift=-4pt]4}] at (4,0.5) {};
\node[label={[below, yshift=-4pt]2}] at (2,0) {};
\node[label={[below, yshift=-4pt]3}] at (3,0) {};
\draw[ultra thick] (2,0) to (3,0);
\draw[ultra thick] (0,0.5) to (4,0.5);
\draw[ultra thick] (1,1) to (5,1);
\end{tikzpicture}
\]
is a regular representation.
\end{ex}

\begin{defi}
If $Z(\m)$ is a regular representation, then one can define a corresponding permutation $\sigma_\m $ as follows. Write $\m = [a_1;b_1] + [a_2;b_2] +\cdots +[a_N;b_N] $, and assume $b_1 > b_2 > \cdots > b_N$, then $\sigma_\m \in\mathfrak{S}_N$ is such that
\[
a_{\sigma_\m(1)} < a_{\sigma_\m(2)} < \cdots  < a_{\sigma_\m(N)}.
\]
\end{defi}

\begin{rem}
If $Z(\m)$ is a ladder representation, then the associated permutation is $w_0$, the longest element of $\mathfrak{S}_N$.
\end{rem}

\begin{defi}
An irreducible representation $\pi$ is said to be \emph{real} if $\pi\times \pi$ is also irreducible.
\end{defi}

\begin{rem}
Real representations are usually called \emph{square-irreducible} representations in this context, but we use \emph{real} here, which is the terminology coming from the work of Kang-Kashiwara-Kim-Oh \citep{KKKO15}  on representations of quantum affine algebras, where the notion appeared in a crucial way (see Section~\ref{sect_cluster}).
\end{rem}

The following is one of the main results of \citep{LM18}.

\begin{theo}\label{theo_pat}
The regular representation $Z(\m)$ is real if and only if there does not exists a sequence $1\leq j_1 < \cdots < j_r  \leq N$, $r\geq 4$ such that if $a'_i=a_{j_i}$ and $b'_i=b_{j_i}$, then either
\begin{align*}
& a'_{i+1} < a'_{i}\leq  b'_{i+1} +1, i=3,\ldots, r-1, a'_3 < a'_1\leq  b'_3 +1,\text{ and } a'_r < a'_2 < a'_{r-1},\\
\text{or} \quad & a'_{i+1} < a'_{i}\leq  b'_{i+1} +1, i=4,\ldots, r-1, a'_4 < a'_2\leq  b'_4 +1,\text{ and } a'_3 < a'_r < a'_1 < a'_\ell, 
\end{align*}
where $\ell=2$ if $r=4$ and $\ell=r-1$ otherwise.
\end{theo}
If the permutation $\sigma_\m$ avoids the patterns $4231$ and $3412$, then the condition of Theorem~\ref{theo_pat} is satisfied. We will call these representations \emph{pattern avoiding regular}.

The same patterns avoidance condition correspond to the smoothness condition of the Schubert variety $X_{\sigma_\m}$ (see \citep{LS90}).

Note that in particular, all ladder representations are real.

\subsection{Irreducibility criteria}

The following result will be much used.

\begin{lem}\label{lem_irr}\citep{MSDuke14}
Let $\pi_1$ and $\pi_2$ be irreducible representations, and $\pi$ be a representation such that
\begin{enumerate}[(a)]
	\item $\pi$ is a subrepresentation of $\pi_1\times \pi_2$,
	\item $\pi$ is a quotient of $\pi_2\times \pi_1$,
	\item $\pi_1\otimes \pi_2$ has multiplicity 1 in the Jordan-Hölder sequence of $\pi_1\times \pi_2$,
\end{enumerate}
Then $\pi$ is irreducible.
\end{lem}

\begin{defi}
Given $\pi_1=Z(\m_1)$ and  $\pi_2=Z(\m_2)$, we write $\LI(\pi_1,\pi_2)$ (resp. $\RI(\pi_1,\pi_2)$) for the condition 
\[Z(\m_1 + \m_2) = \soc(\pi_1\times \pi_2)\]
(resp. 
\[Z(\m_1 + \m_2) = \cs(\pi_1\times \pi_2)).\]
\end{defi}

\begin{lem}\label{lem_equi_LI_RI}
Let $\m$ be a multisegment and $\Delta$ a segment. Then we have the following equivalences:
\begin{align*}
\LI(Z(\Delta),Z(\m)) & \quad \Longleftrightarrow \quad Z(\m+\Delta) \hookrightarrow Z(\Delta)\times Z(\m),\\
\RI(Z(\Delta),Z(\m)) & \quad \Longleftrightarrow \quad Z(\m+\Delta) \hookrightarrow Z(\m) \times Z(\Delta).
\end{align*}
\end{lem}

\begin{proof}
The first statement follows from the fact that the segment representation $Z(\Delta)$ is a \emph{left multiplier} (see \citep[Definition 4.3]{LM16}), thus $Z(\Delta)\times Z(\m)$ has a unique irreducible submodule, which appears with multiplicity 1 in the Jordan-Hölder sequence of $Z(\Delta)\times Z(\m)$.

The second statement can be deduced from the first by the use of the contragredient, or more precisely \citep[Lemma 3.9]{LM16}.
\end{proof}

\begin{prop}\label{prop_irred_LI_RI}\citep{LM16}
$\pi_1\times\pi_2$ is irreducible if and only if $\LI(\pi_1,\pi_2)$ and $\RI(\pi_1,\pi_2)$.
\end{prop}

In \citep{LM16}, Lapid-Minguez introduced a combinatorial setup in order to determine whether the conditions $\RI(Z(\Delta),Z(\m))$ and $\LI(Z(\Delta),Z(\m))$ where satisfied, for $\Delta\in\Seg$ and $\m\in \Mult$. Let us recall it here.

Write $\m=\Delta_1 + \cdots + \Delta_N$, and consider the sets
\begin{align*}
X_{\Delta,\m}  = \left\lbrace i \mid \Delta \prec \Delta_i \right\rbrace, \quad & \tilde{X}_{\Delta,\m}  = \left\lbrace i \mid \Delta_i \prec \Delta \right\rbrace,\\
Y_{\Delta,\m}  = \left\lbrace i \mid \overleftarrow{\Delta} \prec \Delta_i \right\rbrace, \quad & \tilde{Y}_{\Delta,\m}  = \left\lbrace i \mid \overleftarrow{\Delta}_i \prec \Delta \right\rbrace.
\end{align*}

\begin{defi}\label{def_LC_RC}
Let $\LC(\Delta,\m)$ be the condition that there exists an injective function $f : X_{\Delta,\m} \to Y_{\Delta,\m}$ such that for all $1\leq i\leq N$, $\Delta_{f(i)}\prec \Delta_i$.

Let $\RC(\Delta,\m)$ be the condition that there exists an injective function $f : \tilde{X}_{\Delta,\m} \to \tilde{Y}_{\Delta,\m}$ such that for all $1\leq i\leq N$, $\Delta_i\prec \Delta_{f(i)}$.
\end{defi}

The function of Definition \ref{def_LC_RC} are called \emph{matching functions}.

\begin{prop}\citep{LM16}
The conditions $\LC(\Delta,\m)$ and $\LI(Z(\Delta),Z(\m))$ (resp. $\RC(\Delta,\m)$ and $\RI(Z(\Delta),Z(\m))$) are equivalent.
\end{prop}\label{prop_LC_LI}

Combining this result with Proposition \ref{prop_irred_LI_RI}, we get the following.
\begin{cor}
The parabolic induction $Z(\Delta)\times Z(\m)$ is irreducible if and only if $\LC(\Delta,\m)$ and $\RC(\Delta,\m)$.
\end{cor}

\section{Good segments}\label{sect_good}

\subsection{Definition}

\begin{defi}
A segment $\Delta$ in a multisegment $\mathfrak{m}\in\Mult$ is called a \emph{good segment} if 
\begin{enumerate}[(i)]
	\item $Z(\Delta) \times Z(\mathfrak{m}) $ is irreducible.
	\item\label{cas2} $\left\lbrace \begin{array}{r}
	Z(\mathfrak{m}) \hookrightarrow Z(\Delta)\times Z(\mathfrak{m}^-),\\[1ex]
	\text{or } Z(\mathfrak{m}) \hookrightarrow Z(\mathfrak{m}^-)\times Z(\Delta),
	\end{array}\right.$, where $\mathfrak{m}^-=\mathfrak{m}\setminus \{\Delta\}$.
\end{enumerate}
	If the first (resp. second) subcase of (\ref{cas2}) is satisfied, $\Delta$ is called a \emph{good left} (resp. \emph{right}) segment of $\mathfrak{m}$.
\end{defi}

Using Lemma~\ref{lem_equi_LI_RI} as well as Proposition~\ref{prop_LC_LI}, we have the following equivalences:
\begin{align}
\Delta \text{ is a good \textbf{left} segment of } \m & \quad \Longleftrightarrow \quad \begin{array}{l}
 \LC(\Delta,\m)  ,   \RC(\Delta,\m),\\
\text{and }  \LC(\Delta,\m^-).
\end{array} ,\\
\Delta \text{ is a good \textbf{right} segment of } \m & \quad \Longleftrightarrow \quad \begin{array}{l}
 \LC(\Delta,\m)  ,   \RC(\Delta,\m),\\
\text{and }  \RC(\Delta,\m^-).
\end{array}
\end{align}

\subsection{Combinatorial criteria}\label{sect_comb}

\begin{lem}\label{lemLcRCvide}
If $\m=\Delta_1+\cdots+\Delta_N$ is a multisegment and $\Delta_0$ is a segment such that $\Delta_0+\m$ is a \emph{regular} multisegment, then 
\begin{align*}
\LC(\Delta_0,\m) & \Leftrightarrow \nexists i, \Delta_0 \prec \Delta_i,\\
\RC(\Delta_0,\m) & \Leftrightarrow \nexists i , \Delta_i \prec \Delta_0.
\end{align*}
\end{lem}

\begin{proof}
We will prove the first equivalence, the second being exactly analog. From the definition of the condition $\LC$, the implication
\[ \LC(\Delta_0,\m)  \Leftarrow \nexists i, \Delta_0 \prec \Delta_i
\] 
is clear.

Let $i\in Y_{\Delta_0,\m}$. If $i\notin X_{\Delta_0,\m}$, then either $b(\Delta_0)=b(\Delta_i)$ or $e(\Delta_0)=e(\Delta_i)$, which is a contradiction. Thus $Y_{\Delta_0,\m}\subset X_{\Delta,\m}$.
Now, if $X_{\Delta_0,\m} \neq \emptyset$, then $\LC(\Delta_0,\m)$ can not be satisfied (by Hall's marriage theorem).
\end{proof}

\begin{lem}\label{lemLCRCdecr}
If $\m=\Delta_1+\cdots + \Delta_N$ is an ordered regular multisegment with $\sigma =\sigma_\m$ the associated permutation, then for all $1\leq i\leq N$,
\begin{itemize}
	\item[•] the condition $\LC(\Delta_i,\m)$ is equivalent to $\sigma^{-1}$ is strictly decreasing on $X_{\Delta_i,\m}$,
	\item[•] the condition $\RC(\Delta_i,\m)$ is equivalent to $\sigma^{-1}$ is strictly decreasing on $\tilde{X}_{\Delta_i,\m}$.
\end{itemize}
\end{lem}

\begin{proof}
As before, we only prove the first statement. Fix $1\leq i\leq N$. If $X_{\Delta_i,\m} = \emptyset$, the equivalence is trivial. 

Suppose $X_{\Delta_i,\m} \neq \emptyset$, then with the same reasoning as in the proof of Lemma~\ref{lemLcRCvide},
\[Y_{\Delta_i,\m} \subset X_{\Delta_i,\m} \cup \{i\}.\]

Suppose $Y_{\Delta_i,\m} = X_{\Delta_i,\m} \cup \{i\}$. Let $X_{\Delta_i,\m} = \{ j_1 > j_2 > \cdots > j_m \}$. Then, since $\m$ is ordered, $e(\Delta_{j_1}) < e(\Delta_{j_2}) < \cdots < e(\Delta_{j_m})$.

If $\sigma^{-1}$ is strictly decreasing on 
$X_{\Delta_i,\m}$, then $b(\Delta_{j_1}) < b(\Delta_{j_2}) < \cdots < b(\Delta_{j_m})$. Since all $j_k\in X_{\Delta_i,\m}$, we have $\Delta_{j_\ell} \prec \Delta_{j_{\ell+1}}$ for all $1\leq \ell\leq m-1$. Thus the function
\begin{equation}\label{fmatch}
\begin{array}{rrll}
f: & X_{\Delta,\m} & \to Y_{\Delta,\m},\\
& j_1 & \mapsto i,\\
& j_{\ell+1} & \mapsto j_{\ell},& 1\leq \ell\leq m-1,
\end{array}
\end{equation}
is a matching function from $X_{\Delta_i,\m}$ to $Y_{\Delta_i,\m}$. Thus $\LC(\Delta_i,\m)$.

If $Y_{\Delta_i,\m} \subsetneq X_{\Delta_i,\m} \cup \{i\}$, then there exists $j \in X_{\Delta_i,\m}$ such that $b(\Delta_j)=e(\Delta_i)+1$, and $Y_{\Delta_i,\m} = \left(X_{\Delta_i,\m}\setminus \{j\} \right) \cup \{i\}$. If $\sigma^{-1}$ is strictly decreasing on $X_{\Delta,\m}$, then necessarily $j=j_m$ and the function $f$ from \eqref{fmatch} is a matching function from $X_{\Delta_i,\m}$ to $Y_{\Delta_i,\m}$, as $j_m$ does not appear in the image of $f$. 

Conversely, suppose $\LC(\Delta_i,\m)$ and let $f$ be a matching function from $X_{\Delta_i,\m}$ to $Y_{\Delta_i,\m}$. Necessarily, $f(j_1) = i$, as $\Delta_i$ is the only segment considered which precedes $\Delta_{j_1}$. Recursively, we see that $f$ is the function from \eqref{fmatch}. As it is a matching function, we deduce that $\Delta_{j_\ell} \prec \Delta_{j_{\ell+1}}$ for all $1\leq \ell\leq m-1$, and thus $\sigma^{-1}$ is strictly decreasing on $X_{\Delta_i,\m}$.
\end{proof}

\begin{rem}
From Lemma~\ref{lemLcRCvide}, if $\m$ is a regular multisegment, for all $1\leq i\leq k$, $\LC(\Delta_i,\m - \Delta_i)$ (resp. $\RC(\Delta_i,\m - \Delta_i)$) is equivalent to the fact that $\Delta_i$ precedes (resp. is preceded by) no segment in $\m$.

Combining with Lemma~\ref{lemLCRCdecr}, we have the following equivalences:

\begin{tabular}{rcl}
$\Delta$ is a good \textbf{left} segment of $\m$ & $\Longleftrightarrow$ & \begin{minipage}[c]{7.5cm}
$\Delta$ precedes no other segment of $\m$ and $\Delta$ forms a ladder with the segments which precede it.
\end{minipage}\\[15pt]
$\Delta$ is a good \textbf{right} segment of $\m$ & $\Longleftrightarrow$ & \begin{minipage}[c]{7.5cm}
$\Delta$ is preceded by no other segment of $\m$ and $\Delta$ forms a ladder with the segments which are preceded by it.
\end{minipage}
\end{tabular}
\end{rem}

The following result is clear using this criteria.

\begin{lem}\label{lem_good_sub}
If $\m'$ is a sub-multisegment of $\m$ and $\Delta \in \m'$ is a good segment for $\m$, then it is a good segment for $\m'$.
\end{lem}

\begin{rem}
Note that the converse is not true. For example, any segment $\Delta$ is a good segment for itself, but not necessarily a good segment for any multisegment containing it.
\end{rem}

%
%
%
%
%

\subsection{Existence results}\label{sect_existence}

\begin{prop}\label{prop_existence}
For $N\geq 2$, let $\mathfrak{m}=\Delta_1 + \Delta_2 + \cdots + \Delta_N$ be a regular multisegment such that for all $i$, $\Delta_i=[a_i,b_i]$ with $b_1 > b_2 > \cdots > b_N$ and the associated permutation $\sigma$ avoids the patterns $4231$ and $3412$, and $\pi=Z(\mathfrak{m})$ is a \emph{prime} irreducible representation.  Then 
\begin{center}
\begin{tabular}{rccl}
& either $\Delta_1$ & or & $\Delta_{\sigma(1)}$, \\
and & either $\Delta_N$ & or & $\Delta_{\sigma(N)}$
\end{tabular}
\end{center}
correspond to good segments of $\m$.

Moreover, if $\sigma(N)=1$ (resp. $\sigma(1)=N$), then $\Delta_N$ is a good \textbf{right} segment (resp. $\Delta_1$ is a good \textbf{left} segment) of $\m$. If $\sigma(N)=1$ and $\sigma(1)=N$ then $\m$ is a ladder.

\end{prop}

\begin{proof}
First of all, if $\sigma(1)=1$, then $\Delta_1$ is not linked with any other segment of $\m$, and it is both a good left and a good right segment of $\m$.

Let $i_0 = \sigma(1)$ and suppose $i_0 >1$. Suppose neither $\Delta_1$ nor $\Delta_{i_0}$ are good segments. From Lemma~\ref{lemLCRCdecr}, $\sigma_\m^{-1}$ is neither decreasing on $X_{\Delta_1,\m}$ nor on $\tilde{X}_{\Delta_{i_0},\m}$.
We consider different cases.

If there exists $i<j<i_0$ such that $\Delta_i \prec \Delta_1$ and $\Delta_j \prec \Delta_1$, or $\Delta_{i_0} \prec \Delta_i$ and $\Delta_{i_0} \prec \Delta_j$ and $\Delta_j\nprec \Delta_j$, the configuration is the following:

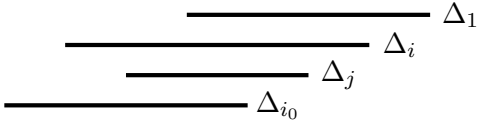
\begin{minipage}[c]{7cm}
\begin{tikzpicture}[scale=0.8]
\draw[ultra thick] (0,0) to (4,0);
\draw[ultra thick] (2,0.5) to (5,0.5);
\draw[ultra thick] (1,1) to (6,1);
\draw[ultra thick] (3,1.5) to (7,1.5);
\node at (4.5,0) {$\Delta_{i_0}$};
\node at (5.5,0.5) {$\Delta_j$};
\node at (6.5,1) {$\Delta_i$};
\node at (7.5,1.5) {$\Delta_1$};
\end{tikzpicture}
\end{minipage}
\begin{minipage}[c]{7.5cm}
The pattern $4231$ appears in this configuration, which is a contradiction.
\end{minipage}

Otherwise, there exists at least one $1<i<i_0$ such that $\Delta_{i_0} \prec \Delta_i$ and $\Delta_i \nprec \Delta_1$ and one $i_0 <j$ such that $\Delta_j\prec \Delta_1$ and $\Delta_j \nprec \Delta_{i_0}$.  The configuration is the following:

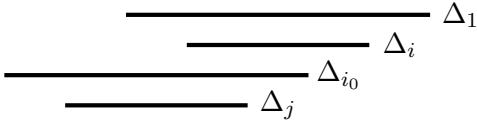
\begin{minipage}[c]{7cm}
\begin{tikzpicture}[scale=0.8]
\draw[ultra thick] (1,0) to (4,0);
\draw[ultra thick] (0,0.5) to (5,0.5);
\draw[ultra thick] (3,1) to (6,1);
\draw[ultra thick] (2,1.5) to (7,1.5);
\node at (4.5,0) {$\Delta_{j}$};
\node at (5.5,0.5) {$\Delta_{i_0}$};
\node at (6.5,1) {$\Delta_i$};
\node at (7.5,1.5) {$\Delta_1$};
\end{tikzpicture}
\end{minipage}
\begin{minipage}[c]{7.5cm}
The pattern $3412$ appears in this configuration, which is a contradiction.
\end{minipage}

The proof for $\Delta_N$ and $\Delta_{\sigma(N)}$ is exactly symmetric.

Now, suppose $\sigma(N)=1$. We know that either $\Delta_1$ or $\Delta_N$ is a good segment of $\m$. If $\Delta_1$ is a good segment and $\Delta_N$ is not a good segment, then we are necessarily in the first configuration drawn above, which features the pattern $4231$. 
Similarly, if $\sigma(1)=N$, then either $\Delta_1$ or $\Delta_N$ is a good segment of $\m$. The same pattern avoidance condition implies that $\Delta_1$ is necessarily good. 

If both $\sigma(N)=1$ and $\sigma(1)=N$ then any pair of segments between $\Delta_1$ and $\Delta_N$ which does not form a ladder would create a $4231$ pattern. Thus $\m$ is a ladder.
\end{proof}

%

This result has the following consequence.

\begin{cor}\label{cor_2good}
Let $\pi=Z(\mathfrak{m})$ be a regular representation avoiding the patterns $4231$ and $3412$, then $\pi$ has at least two good segments.
\end{cor}

\begin{rem}\label{rem_avoid_real}
This criteria allows us to recover the implication (which is established by Theorem~\ref{theo_pat} \cite{LM18}):
\begin{center}
$\m$ avoids the patterns $4231$ and $3412$ \quad $\Longrightarrow$ \quad $Z(\m)$ is real.
\end{center}
This can be proved by induction on $N$, the number of segments in the multisegment $\m$. For completeness, let us detail the reasoning.

If $N=1$, then $\m=\Delta$ is just a segment, and $Z(\Delta)$ is real (for example as an application of Proposition~\ref{prop_notprime}). 

If $N\geq 2$ and $\m$ avoids the patterns $4231$ and $3412$, then from Corollary~\ref{cor_2good}, $\m$ has at least \textbf{one} good segment $\Delta$. Suppose without loss of generality that it is a good left segment. Then
\begin{align*}
Z(\m)\times Z(\m) & \hookrightarrow \underbrace{Z(\m)\times Z(\Delta)}_{\text{irreducible}}\times Z(\m^-),\\
&  \hookrightarrow Z(\Delta)\times Z(\m)\times Z(\m^-) \hookrightarrow \underbrace{Z(\Delta)\times Z(\Delta)}_{\text{irreducible}}\times Z(\m^-)\times Z(\m^-).
\end{align*}
However, $\m-$ has $N-1$ segments, and satisfy the pattern avoidance condition. Thus $ Z(\m^-)\times Z(\m^-)$ is irreducible by induction hypothesis.

Similarly, as $Z(\m)\twoheadleftarrow  Z(\m^-)\times Z(\Delta)$,
\[Z(\m)\times Z(\m)  \twoheadleftarrow Z(\m^-)\times Z(\m^-) \times Z(\Delta)\times Z(\Delta). \] 
Then the irreducibility of $Z(\m)\times Z(\m) $ is obtained through Lemma~\ref{lem_irr}.

Notice we only used the existence of one good segment in the proof, although there is two from Corollary~\ref{cor_2good}. 
\end{rem}

\section{Determinant formula}\label{sect_det_formula}

	\subsection{Alternate sum formula}
	
One of the results of \citep{LM18} is an alternate sum formula for every regular real representation using standard representations. Let $\pi=Z(\m)$ be a regular real representation, with $\m=[a_1;b_1] + \cdots + [a_N;b_N]$ such that $b_1 > \cdots > b_N$.

In the Grothendieck ring,
\begin{equation}\label{eq_alt_sum}
\pi = \sum_{\substack{\sigma'\in S_N \\ \sigma_0 \leq \sigma'\leq \sigma}} \sgn(\sigma'\sigma) Z([a_{\sigma(1)};b_{\sigma'(1)}])\times Z([a_{\sigma(2)};b_{\sigma'(2)}])\times \cdots \times Z([a_{\sigma(N)};b_{\sigma'(N)}]),
\end{equation}	
where $\sigma = \sigma_\m$ and for all $i$,
\[\sigma_0^{-1}(i)=\max\left\lbrace j\leq x_i \mid j\notin \sigma_0^{-1}(\{i+1, \ldots,N\})\right\rbrace, \quad \text{with } x_i =\#\{j\mid a_j\leq b_i+1\}.\]

\begin{rem}
The permutation $\sigma_0$ satisfies 
\[ \sigma_0 \leq \sigma' \quad \Leftrightarrow \quad \forall  i\in\{1,\ldots, N\}, a_{\sigma(i)}\leq b_{\sigma'(i)} +1.\]
\end{rem}

We deduce that equation~\eqref{eq_alt_sum} is equivalent to
\begin{equation}\label{eq_alt_sum_2}
\pi = \sum_{\substack{\sigma'\in S_N \\ \sigma'\leq \sigma}} \sgn(\sigma'\sigma) Z([a_{\sigma(1)};b_{\sigma'(1)}])\times Z([a_{\sigma(2)};b_{\sigma'(2)}])\times \cdots \times Z([a_{\sigma(N)};b_{\sigma'(N)}]).
\end{equation}
Indeed, for $\sigma'>\sigma_0$, at least one of the 	$Z([a_{\sigma(i)},b_{\sigma'(i)}])$ is not defined, and the term does not contribute to the sum in \eqref{eq_alt_sum_2}.

For all $i\in\{1,\ldots,N\}$, set $a'_i=a_{\sigma(i)}$, then equation~\eqref{eq_alt_sum_2} can be rewritten
\begin{equation}\label{eq_alt_sum_3}
\pi = \sgn(\sigma)\sum_{\sigma'\in[\id,\sigma]} \sgn(\sigma') Z([a'_1;b_{\sigma'(1)}])\times Z([a'_2;b_{\sigma'(2)}])\times \cdots \times Z([a'_N;b_{\sigma'(N)}]),
\end{equation}
where $[\id,\sigma]$ denotes the Bruhat interval of permutations in $S_N$ lower than $\sigma$.

	\subsection{Matrix determinant}

Equation~\eqref{eq_alt_sum_3} is similar to the determinant of a matrix, with some entries replaced by zeros to account for the missing permutations $\sigma'$. More precisely, for $\sigma_1,\sigma_2$ permutations in $S_N$, let
\[ \Gamma[\sigma_1, \sigma_2] := \{ (i,\sigma(i)) \mid \sigma\in [\sigma_1,\sigma_2], 1\leq i\leq N\}, \]
then permutations whose graph is contained in $\Gamma[\id, \sigma]$ form the \emph{right convex hull}, from the work of Sjöstrand \cite{S07}. The following is obtained using \cite[Theorem 4]{S07}.

\begin{prop}\cite[Proposition 3.3]{CSB21} \label{prop_RCH}
If the permutation $\sigma\in S_N$ avoids the patterns $4231$ and $3412$\footnote{in \cite{S07,CSB21}, the pattern avoidance condition is weaker, permutations are assumed to avoid the patterns $4231$, $35142$, $42513$, and $351624$}, and $M=(m_{i,j})_{1\leq i,j\leq N}$ is a square $N\times N$-matrix, then
\begin{equation}\label{eq_det_M_Gamma}
\det(M|_{\Gamma[\id, \sigma]}) = \sum_{\sigma'\in[\id,\sigma]}\sgn(\sigma') m_{1,\sigma'(1)}m_{2,\sigma'(2)}\cdots m_{N,\sigma'(N)}.
\end{equation}
\end{prop}

\begin{rem}
Using Ferrers boards (see Appendix~\ref{sect_Ferrers}), the determinant in equation \eqref{eq_det_M_Gamma} can be computed placing the coefficient $m_{i,j}$ in the box $(i,j)$ of $[N]^2$ if it is coloured and 0 if it is not. Note that the dots are placed on the $Z([a_i;b_i])$. 
\end{rem}

Combining Proposition~\ref{prop_RCH} with \eqref{eq_alt_sum_3}, and assuming $\sigma$ avoids the patterns $4231$ and $3412$, we obtain the following:
\begin{equation}\label{eq_pi_matrix}
\pi = \sgn(\sigma)\det\left( (Z(a'_i;b_j))_{1\leq i,j\leq N} |_{\Gamma[\id, \sigma]}\right).
\end{equation}

	\subsection{Lewis Carroll's identity}

The following result is usually called the \emph{Lewis Carroll's identity} or the \emph{Desnanot–Jacobi identity}.

\begin{prop}\label{prop_lewis_carroll}
For $M$ a square $N\times N$-matrix, and $A,B \subset \{1,\ldots,N\}$, let $M_A^B$ be the matrix obtained from $M$ by removing all rows indexed by elements of $A$ and all columns indexed by elements of $B$. Then, for all $1\leq a < a' \leq N$ and $1\leq b < b' \leq N$,
\begin{equation}\label{eq_Lewis_Carroll}
\det(M)\det(M_{a,a'}^{b,b'}) = \det(M_{a}^{b})\det(M_{a'}^{b'}) - \det(M_{a}^{b'})\det(M_{a'}^{b}).
\end{equation}
\end{prop}

We can use Proposition~\ref{prop_lewis_carroll} with equation \eqref{eq_pi_matrix} to write relations in the Grothendieck ring $\mathcal{R}$ involving $\pi$. However, if $M=\left( (Z(a'_i;b_j))_{1\leq i,j\leq N} |_{\Gamma[\id, \sigma]}\right)$, the determinant of the submatrix $M_i^j$ does not necessarily realize (up to a sign) the class of an irreducible representation $Z(\m')$ in $\mathcal{R}$, for all $1\leq i,j \leq N$. 

\begin{ex}
Let $\m=[1;4]+[0;3]+[2;2]$, the corresponding permutation is the reflection $\sigma=(12) \in S_3$. The alternate sum formula for the class of the irreducible representation is
\begin{align*}
Z(\m)& = Z([1;4])\times Z([0;3]) \times Z([2;2]) -Z([0;4])\times Z([1;3]) \times Z([2;2]),\\
	& = - \begin{vmatrix}
	Z([0;4]) & Z([0;3]) & 0 \\
	Z([1;4]) & Z([1;3]) & 0\\
	0 & 0 & Z([2;2])
	\end{vmatrix}.
\end{align*}
Let $M$ be the above matrix, then
\begin{align*}
\det(M_2^1) &= Z([0;3])\times Z([2;2]) = Z([0;3]+[2;2]) \quad \in \mathcal{R},\\
\det(M_3^1) &= 0 \neq Z(\m').
\end{align*}
\end{ex}

Nevertheless, it is possible to write explicit formulas in the Grothendieck ring $\mathcal{R}$ in some interesting cases.

We will use the following key result.

\begin{prop}\cite[Proposition 4.17]{CSB21}\label{prop_417}
Let $\sigma$ be a permutation in $S_N$ avoiding the patterns $4231$ and $3412$, and choose $i\in [N]$. Let $\overline{\sigma}^{i} \in S_{N-1}$ be the "flatten" permutation obtained from $\sigma$ by removing $(i,\sigma(i))$ and shifting the remaining numbers appropriately. Then for $M$ a $(N-1)\times(N-1)$-matrix,
\begin{equation}
\det(M|_{\Gamma[\id, \sigma]_i^{\sigma(i)}}) = \det(M|_{\Gamma[\id, \overline{\sigma}^{i}]}).
\end{equation}
\end{prop}

\begin{rem}
Note for $M$ a $N\times N$-matrix and for $1\leq i,j \leq N$,
\[ \left( M|_{\Gamma[\id, \sigma]}\right)_i^j = M_i^j|_{\Gamma[\id, \sigma]_i^j}. \]
\end{rem}

\section{Extended \texorpdfstring{$T$}{T}-system formula}\label{sect_main}

Our main result is the following, which will be proven in Section~\ref{proof_rel} and \ref{proof_irred}.
\begin{theo}\label{theo_main}
Let $\m = \Delta_1 + \Delta_2 + \cdots +\Delta_N$ be a regular multisegment, such that $b_1 > b_2 > \cdots > b_N$, where for all $1\leq i\leq N$, $\Delta_i=[a_i;b_i]$. Assume the corresponding permutation $\sigma$ avoids the patterns $4231$ and $3412$, and that $\sigma(N)\neq N$. Let 
\begin{equation*}
I =\left\lbrace i \mid \begin{array}{l}
a_N \leq a_i \\
b_i \leq b_{\sigma(N)}
\end{array}\right\rbrace = \{i_1 < i_2 < \cdots i_r\}.
\end{equation*}
Then, we have the following relation, in the Grothendieck ring $\mathcal{R}$:
\begin{equation}
Z(\m\setminus \Delta_N)\times Z(\m\setminus \Delta_{\sigma(N)}) = Z(\m)\times Z(\m\setminus \Delta_N,\Delta_{\sigma(N)}) + Z(\m')\times Z(\m''),\label{eq_moins_N}
\end{equation}
where 
\begin{equation*}
\m'=\sum_{j\notin I}\Delta_j + \sum_{k=1}^{r-1}[a_{i_k};b_{i_{k+1}}], \quad \m''=\sum_{i\notin I}\Delta_i + \sum_{k=1}^{r-1}[a_{i_{k+1}};b_{i_{k}}].
\end{equation*}
Moreover, the products in both terms on the right hand side of \eqref{eq_moins_N} are irreducible.
\end{theo}

\begin{rem}\begin{enumerate}
	\item If $\sigma(N)=N$, then the segment $\Delta_N$ is not linked to any other segment of $\m$. In that case
\[ Z(\m) = Z(\m\setminus \Delta_N)\times Z(\Delta_N).\]

	\item As $\sigma$ avoids the pattern $4231$, the segments $\Delta_i$, with $i\in I$ form a ladder.
\end{enumerate}
\end{rem}

\begin{cor}\label{cor_main}
Let us assume the permutation $\sigma$ avoids the patterns $4231$ and $3412$ and satisfies $\sigma(1)\neq 1$. Let
\[J=\left\lbrace j \mid \begin{array}{l}
 a_j  \leq a_1\\
b_{\sigma(1)} \leq b_j 
\end{array}\right\rbrace  = \{j_1 < j_2 < \cdots j_s\}. \]
The following relation in satisfied in the Grothendieck ring $\mathcal{R}$:
\begin{equation}
Z(\m\setminus \Delta_1)\times Z(\m\setminus \Delta_{\sigma(1)}) = Z(\m)\times Z(\m\setminus \Delta_1,\Delta_{\sigma(1)}) + Z(\m')\times Z(\m''),\label{eq_moins_1}
\end{equation}
where
\[\m'=\sum_{i\notin J}\Delta_i + \sum_{k=1}^{s-1}[a_{j_k};b_{j_{k+1}}], \quad \m''=\sum_{i\notin J}\Delta_i + \sum_{k=1}^{s-1}[a_{j_{k+1}};b_{j_{k}}].\]
\end{cor}

\begin{proof}
The result is obtained by applying Theorem~\ref{theo_main} to the irreducible representation $Z(\m')$, with $\m' = [-b_N;-a_N] + \cdots + [-b_1;-a_1]$.
\end{proof}

\begin{ex}\label{ex_segs}
\begin{enumerate}
	\item Let $\m = [2;3] + [0;2] + [ 1;1]$. The corresponding regular  representation  $Z(\m)$ is real, since its associated permutation is $\sigma = 231$. It has two good right segments, which are $[0;2]$ and $[ 1;1]$ ($\Delta_{\sigma(1)}$ and $\Delta_3$). Applying Theorem~\ref{theo_main} gives the following relation:
	\[ Z([2;3] + [0;2])\times Z([0;2] + [ 1;1]) = Z(\m)\times Z([0;2]) + Z([0;2])\times Z([1;3] + [0;2]).
	\]
Note that $Z([0;2] + [ 1;1]) \cong Z([0;2]) \times Z([ 1;1])$, and in this case the above relation can be simplified by $Z([0;2])$.

	\item Let $\m = [1;6] + [3;5] + [0;4] + [ 2;3]$. The corresponding regular  representation  $Z(\m)$ is real, since its associated permutation is $\sigma = 3142$. It has two good segments, which are $[1;6]$ (left) and $[ 2;3]$ (right). Applying Theorem~\ref{theo_main} gives the following relation:
\begin{multline*}
Z([1;6] + [3;5] + [0;4])\times Z([1;6] + [0;4] + [ 2;3]) = Z(\m)\times Z([1;6] + [0;4]) \\
+ Z([1;6] + [0;4] + [3;3])\times Z([1;6] + [0;4] + [2;5]).
\end{multline*}
Whereas applying Corollary~\ref{cor_main} gives the following relation:
\begin{multline*}
Z([3;5] + [0;4] + [2;3])\times Z([1;6] + [3;5] + [2;3]) = Z(\m)\times Z([3;5] + [2;3]) \\
+ Z([3;5] + [2;3] + [1;4])\times Z([3;5] + [2;3] + [0;6]).
\end{multline*}
\end{enumerate}
\end{ex}

	\subsection{Ladder case}\label{sect_ladders}

If $\m$ is a ladder, then the corresponding permutation is the longest permutation $w_0$. Thus $\Gamma[\id, w_0]=[N]$. In that case, the result of Theorem~\ref{theo_main} is already known, as Corollary~12 of \cite{LM14}, or Theorem~4.1 in \cite{MY12}, in the language of representations of quantum affine algebras.

\begin{theo}\label{theo_ladder}
Let $\m = \Delta_1+\cdots+\Delta_N$ be a ladder multisegment, with $\Delta_i=[a_i;b_i]$, then
\begin{multline}\label{eq_ext_T-sys}
Z(\Delta_1 + \cdots + \Delta_{N-1})\times Z(\Delta_2+ \cdots + \Delta_{N}) = Z(\m)\times Z(\Delta_2 + \cdots + \Delta_{N-1}) \\
+ Z([a_1;b_2] + \cdots + [a_{N-1};b_N])\times Z([a_2;b_1] + \cdots + [a_{N};b_{N-1}]).
\end{multline}
\end{theo}

In this case, the result comes from the application of the Lewis Carroll identity \eqref{eq_Lewis_Carroll} to the matrix $\left( Z([a_i;b_j])\right)_{1\leq i,j\leq N}$, on lines and columns 1 and $N$. However, in order to understand better the general case, let us consider in more details the application of the Lewis Carroll identity to the matrix $M=\left( Z([a'_i;b_j])\right)_{1\leq i,j\leq N}$ (recall that $a'_i=a_{N-i+1}$). 

One can look at what happens to the Ferrers boards (see Appendix~\ref{sect_Ferrers}) in this case . The permutation $w_0$ is represented by an anti-diagonal, and the Ferrers board is the full grid. Taking out row 1 and column $N$, one gets exactly the grid corresponding to the longest element of $S_{N-1}$. 

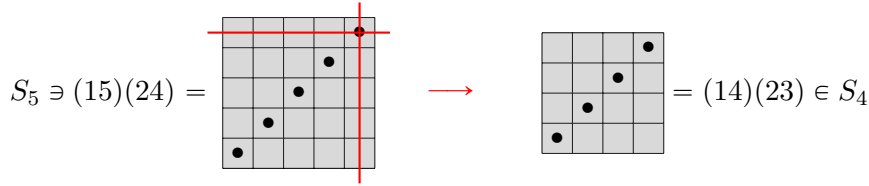
\begin{center}
\begin{tikzpicture}
\node at (-2.5,0) {$S_5 \ni (15)(24) =$};
\node at (0,0) {
\begin{tikzpicture}
\draw[step=0.4cm,color=black] (0,0) grid (2,2);
\fill[color=gray, opacity=0.3] (0,0) rectangle (2,2);
\node at (1.8,1.8) {$\bullet$};
\node at (1.4,1.4) {$\bullet$};
\node at (1,1) {$\bullet$};
\node at (0.6,0.6) {$\bullet$};
\node at (0.2,0.2) {$\bullet$};
\draw[thick, red] (-0.2,1.8) to (2.2,1.8);
\draw[thick, red] (1.8,-0.2) to (1.8,2.2);
\end{tikzpicture}};
\node[red,thick] at (2,0) {$\longrightarrow$};
\node at (4,0) {
\begin{tikzpicture}
\draw[step=0.4cm,color=black] (0,0) grid (1.6,1.6);
\fill[color=gray, opacity=0.3] (0,0) rectangle (1.6,1.6);
\node at (1.4,1.4) {$\bullet$};
\node at (1,1) {$\bullet$};
\node at (0.6,0.6) {$\bullet$};
\node at (0.2,0.2) {$\bullet$};
\end{tikzpicture}};
\node at (6.2,0) {$= (14)(23) \in S_4$};
\end{tikzpicture}
\end{center}

As the signature of the longest permutation in $S_N$ is $(-1)^{\lfloor \frac{N}{2}\rfloor}$, one has
\begin{align*}
(-1)^{\lfloor \frac{N-1}{2}\rfloor}\det(M_N^1) & = Z(\Delta_1 + \cdots + \Delta_{N-1}),\\
(-1)^{\lfloor \frac{N-1}{2}\rfloor}\det(M^N_1) & = Z(\Delta_2 + \cdots + \Delta_{N}),\\
(-1)^{\lfloor \frac{N}{2}\rfloor-1}\det(M_{1,N}^{1,N}) & = Z(\Delta_2 + \cdots + \Delta_{N-1}).
\end{align*}

Now, taking out row 1 and column 1, or row $N$ and column $N$, one gets again the grid corresponding to the longest element of $S_{N-1}$, but the dots have moved. For example, if one does a cyclic permutation of the columns by shifting them to the left and placing column 1 at the end, then taking out row 1 and column 1 gives the same result as taking out row 1 and column $N$ in the shifted board. 

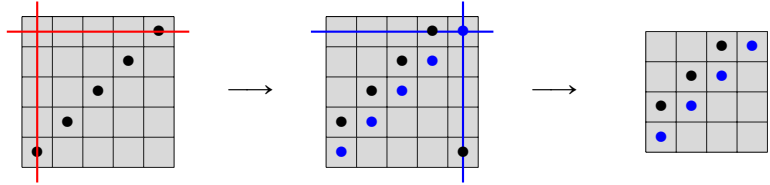
\begin{center}
\begin{tikzpicture}
\node at (0,0) {
\begin{tikzpicture}
\draw[step=0.4cm,color=black] (0,0) grid (2,2);
\fill[color=gray, opacity=0.3] (0,0) rectangle (2,2);
\node at (1.8,1.8) {$\bullet$};
\node at (1.4,1.4) {$\bullet$};
\node at (1,1) {$\bullet$};
\node at (0.6,0.6) {$\bullet$};
\node at (0.2,0.2) {$\bullet$};
\draw[thick, red] (-0.2,1.8) to (2.2,1.8);
\draw[thick, red] (0.2,-0.2) to (0.2,2.2);
\end{tikzpicture}};
\node[thick] at (2,0) {$\longrightarrow$};
\node at (4,0) {
\begin{tikzpicture}
\draw[step=0.4cm,color=black] (0,0) grid (2,2);
\fill[color=gray, opacity=0.3] (0,0) rectangle (2,2);
\node[blue] at (1.8,1.8) {$\bullet$};
\node[blue] at (1.4,1.4) {$\bullet$};
\node[blue] at (1,1) {$\bullet$};
\node[blue] at (0.6,0.6) {$\bullet$};
\node[blue] at (0.2,0.2) {$\bullet$};
\draw[thick, blue] (-0.2,1.8) to (2.2,1.8);
\draw[thick, blue] (1.8,-0.2) to (1.8,2.2);
\node at (1.4,1.8) {$\bullet$};
\node at (1,1.4) {$\bullet$};
\node at (0.6,1) {$\bullet$};
\node at (0.2,0.6) {$\bullet$};
\node at (1.8,0.2) {$\bullet$};
\end{tikzpicture}};
\node[thick] at (6,0) {$\longrightarrow$};
\node at (8,0) {
\begin{tikzpicture}
\draw[step=0.4cm,color=black] (0,0) grid (1.6,1.6);
\fill[color=gray, opacity=0.3] (0,0) rectangle (1.6,1.6);
\node[blue] at (1.4,1.4) {$\bullet$};
\node[blue] at (1,1) {$\bullet$};
\node[blue] at (0.6,0.6) {$\bullet$};
\node[blue] at (0.2,0.2) {$\bullet$};
\node at (1,1.4) {$\bullet$};
\node at (0.6,1) {$\bullet$};
\node at (0.2,0.6) {$\bullet$};
\end{tikzpicture}};
\end{tikzpicture}
\end{center}

The same operation can be applied to the matrix $M$. Note that the new dots are placed on the coefficients $Z([a_1;b_2]),\ldots,Z([a_{N-1};b_N])$. The permutation of the columns does not change the sign of the determinant because the columns on which the determinant is computed are not permuted with respect to one another. 
\[ \det(M_1^1)=\det(\text{shift}(M)_1^N) = (-1)^{\lfloor \frac{N-1}{2}\rfloor}Z([a_1;b_2] + \cdots + [a_{N-1};b_N]). \]

Similarly,
\[
\det(M^N_N)  = (-1)^{\lfloor \frac{N-1}{2}\rfloor}Z([a_2;b_1] + \cdots + [a_{N};b_{N-1}]).\]

Finally, the Lewis Carroll identity \eqref{eq_Lewis_Carroll} gives relation \eqref{eq_moins_N} of Theorem~\ref{theo_ladder}. 

Moreover, the irreductibility of the terms $Z(\m)\times Z(\Delta_2 + \cdots + \Delta_{N-1})$ and $Z([a_1;b_2] + \cdots + [a_{N-1};b_N])\times Z([a_2;b_1] + \cdots + [a_{N};b_{N-1}])$ is proven in \cite[Exemple 4.5]{BLM}.

	\subsection{Proof of relation \eqref{eq_moins_N}}\label{proof_rel}

Let us apply Proposition~\ref{prop_lewis_carroll} (Lewis Carroll's identity) to the matrix $\tilde{M}=M|_{\Gamma[\id, \sigma]}$, where $M=\left( (Z(a'_i;b_j))_{1\leq i,j\leq N} \right)$, on rows $\sigma^{-1}(N),N$ and columns $\sigma(N),N$:
\begin{equation}\label{eq_LC_M}
\det(\tilde{M})\det(\tilde{M}_{\sigma^{-1}(N),N}^{\sigma(N),N}) = \det(\tilde{M}_{\sigma^{-1}(N)}^{\sigma(N)})\det(\tilde{M}_{N}^{N}) - \det(\tilde{M}_{\sigma^{-1}(N)}^{N})\det(\tilde{M}_{N}^{\sigma(N)}).
\end{equation}
Using Proposition~\ref{prop_417}, 
\begin{align*}
\det\left(\tilde{M}_{N}^{\sigma(N)}\right) & = \det\left(M_{N}^{\sigma(N)}|_{\Gamma[\id, \overline{\sigma}^{N}]}  \right), \\
\det\left(\tilde{M}_{\sigma^{-1}(N)}^{N}\right) & = \det\left(M_{\sigma^{-1}(N)}^{N}|_{\Gamma[\id, \overline{\sigma}^{\sigma^{-1}(N)}]}  \right).
\end{align*}
Since $\overline{\sigma}^{N}$ and $\overline{\sigma}^{\sigma^{-1}(N)}$ satisfy the pattern avoidance condition, using \eqref{eq_pi_matrix}, one has
\begin{align*}
\det\left(M_{N}^{\sigma(N)}|_{\Gamma[\id, \overline{\sigma}^{N}]}  \right) & = \sgn(\overline{\sigma}^{N}) Z(\Delta_1 +  \cdots + \widehat{\Delta_{\sigma(N)}}+ \cdots + \Delta_{N}),\\
\det\left(M_{\sigma^{-1}(N)}^{N}|_{\Gamma[\id, \overline{\sigma}^{\sigma^{-1}(N)}]}  \right) & = \sgn(\overline{\sigma}^{\sigma^{-1}(N)}) Z(\Delta_1 + \Delta_2+ \cdots + \Delta_{N-1}).
\end{align*}
Note that for $i\in[N]$, $M_i^{\sigma(i)}$ is obtained by taking out the row and column containing the coefficient $Z([a'_i;b_{\sigma(i)}])= Z(\Delta_{\sigma(i)})$.

Similarly,
\[ \det \left( \tilde{M}_{\sigma^{-1}(N),N}^{\sigma(N),N}\right) = \sgn(\overline{\overline{\sigma}^{\sigma^{-1}(N)}}^N)Z(\Delta_1 +  \cdots + \widehat{\Delta_{\sigma(N)}}+ \cdots + \Delta_{N-1}).  \]

Let us now consider the coefficients $\det(\tilde{M}_{N}^{N})$ and $\det(\tilde{M}_{\sigma^{-1}(N)}^{\sigma(N)})$. Using Lemma~\ref{lem_ferrers}, we know that either the two columns $\sigma(N)$ and $N$ or the two rows $\sigma^{-1}(N)$ and $N$ of $\tilde{M}$ are similar (the zeros are at the same place). 
Assume that the two columns $\sigma(N)$ and $N$ of $\tilde{M}$ are similar, meaning that there are as many zeros above the coefficient $Z([a'_{\sigma^{-1}(N)};b_{\sigma(N)}])$ as above  $Z([a'_{\sigma^{-1}(N)};b_N])$. Note that since $\sigma$ avoids the pattern $4231$, the dots in the lower right corner of the Ferrers board form a ladder. In particular, one can apply the same reasoning as in the ladder case. 

Let us cyclically permute the columns $\sigma(N), \ldots,N$ in order to obtain column $N$ in position $\sigma(N)$, and take the determinant of the minor $\text{shift}(\tilde{M}_N^N) = \text{shift}(\tilde{M})_N^{\sigma(N)}$ instead of the minor $\tilde{M}_N^N$. Because theses columns have the same zero block in their upper part, these determinants are equal. The resulting permutation is $\overline{\sigma}^N$ (same permutation as if we had taken out row $N$ and column $\sigma(N)$).

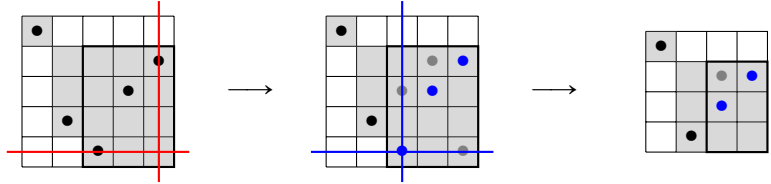
\begin{center}
\begin{tikzpicture}
\node at (0,0) {
\begin{tikzpicture}
\draw[step=0.4cm,color=black] (0,0) grid (2,2);
\fill[color=gray, opacity=0.3] (0,1.6) rectangle (0.4,2);
\fill[color=gray, opacity=0.3] (0.4,0.4) rectangle (2,1.6);
\fill[color=gray, opacity=0.3] (0.8,0) rectangle (2,0.4);
\node at (0.2,1.8) {$\bullet$};
\node at (1.8,1.4) {$\bullet$};
\node at (1.4,1) {$\bullet$};
\node at (0.6,0.6) {$\bullet$};
\node at (1,0.2) {$\bullet$};
\draw[thick] (0.8,0) rectangle (2,1.6);
\draw[thick, red] (-0.2,0.2) to (2.2,0.2);
\draw[thick, red] (1.8,-0.2) to (1.8,2.2);
\end{tikzpicture}
};
\node at (2,0) {$\longrightarrow$};
\node at (4,0) {
\begin{tikzpicture}
\draw[step=0.4cm,color=black] (0,0) grid (2,2);
\fill[color=gray, opacity=0.3] (0,1.6) rectangle (0.4,2);
\fill[color=gray, opacity=0.3] (0.4,0.4) rectangle (2,1.6);
\fill[color=gray, opacity=0.3] (0.8,0) rectangle (2,0.4);
\node at (0.2,1.8) {$\bullet$};
\node[blue] at (1.8,1.4) {$\bullet$};
\node[blue] at (1.4,1) {$\bullet$};
\node at (0.6,0.6) {$\bullet$};
\node[blue] at (1,0.2) {$\bullet$};
\draw[thick] (0.8,0) rectangle (2,1.6);
\node[gray] at (1.4,1.4) {$\bullet$};
\node[gray] at (1,1) {$\bullet$};
\node[gray] at (1.8,0.2) {$\bullet$};
\draw[thick, blue] (-0.2,0.2) to (2.2,0.2);
\draw[thick, blue] (1,-0.2) to (1,2.2);
\end{tikzpicture}
};
\node at (6,0) {$\longrightarrow$};
\node at (8,0) {
\begin{tikzpicture}
\draw[step=0.4cm,color=black] (0,0) grid (1.6,1.6);
\fill[color=gray, opacity=0.3] (0,1.2) rectangle (0.4,1.6);
\fill[color=gray, opacity=0.3] (0.4,0) rectangle (1.6,1.2);
\node at (0.2,1.4) {$\bullet$};
\node[blue] at (1.4,1) {$\bullet$};
\node[blue] at (1,0.6) {$\bullet$};
\node at (0.6,0.2) {$\bullet$};
\draw[thick] (0.8,0) rectangle (1.6,1.2);
\node[gray] at (1,1) {$\bullet$};
\end{tikzpicture}
};
\end{tikzpicture}
\end{center}

For $j\notin I$, there are still dots placed on the $Z([a_j;b_j])$, and the new dots are placed on the $Z([a_{i_{k+1}};b_{i_k}])$, for $1\leq k \leq r$.

Hence,
\[\det(\tilde{M}_N^N)= \det(\text{shift}(\tilde{M})_N^{\sigma(N)}) = \sgn(\overline{\sigma}^N)Z(\m'').\]
Similarly, 
\[\det(\tilde{M}_{\sigma^{-1}(N)}^{\sigma(N)})=  \sgn(\overline{\sigma}^{\sigma^{-1}(N)})Z(\m').\]

Let us now consider the signatures of the permutations, using the criteria of Lemma~\ref{lem_sgn}: 
\[\sgn(\sigma) = \sum_{\bullet}\sharp\left\lbrace\begin{tikzpicture}[scale=0.8,baseline=(AA.base)]
\draw[step=0.4cm,color=black] (0,0) grid (1.2,0.8);
\fill[color=gray, opacity=0.3] (0,0) rectangle (1.2,0.8);
\node[blue] at (1,0.6) {$\bullet$};
\node  at (0.2,0.2) {$\bullet$};
\node (AA) at (0.6,0.3) {};
\end{tikzpicture}
\right\rbrace.\]

Let us separate the grid (or matrix) in 4 blocks, $A,B,C$ and the upper right being empty (or filled with zeros). 
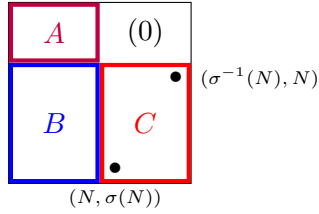
\begin{center}
\begin{tikzpicture}
\draw (0,0) rectangle (2.4,2.4);
\node at (1.8,2) {$(0)$};
\draw[ultra thick, purple] (0.03,1.63) rectangle (1.17, 2.37);
\node[purple] at (0.6,2) {$A$};
\draw[ultra thick, blue] (0.03,0.03) rectangle (1.17, 1.57);
\node[blue] at (0.6,0.8) {$B$};
\draw[ultra thick, red] (1.23,0.03) rectangle (2.37, 1.57);
\node[red] at (1.8,0.8) {$C$};
\node at (1.4,0.2) {$\bullet$};
\node at (1.4,-0.2) {\tiny $(N,\sigma(N))$};
\node at (2.2,1.4) {$\bullet$};
\node at (3.3,1.4) {\tiny $(\sigma^{-1}(N),N)$};
\end{tikzpicture}
\end{center}
Note that zone $C$ contains $r$ dots, where $r$ is the cardinal of $I$. 
Going from $\sigma$ to $\overline{\sigma}^N$, we take out the dot on the last line. It is clear that all dots in zones $A$, $B$ and $C$ except the bottom one will have the same contribution to the sum $\ell$. The difference is thus equal to the contribution of the bottom dot $(N,\sigma(N))$, which is $r-1$. Hence
\[\sgn(\overline{\sigma}^N) = \sgn(\sigma)(-1)^{r-1}.\]

Now, going from $\sigma$ to $\overline{\sigma}^{\sigma^{-1}(N)}$, we take out the dot $(\sigma^{-1}(N),N)$. All dots is block $A$ will still have the same contribution, but the dots in block $B$ will count one less dot in their upper right corner. The remaining dots in block $C$ will also count one less dot. Thus,
\[\sgn(\overline{\sigma}^{\sigma^{-1}(N)}) = \sgn(\sigma)(-1)^{r-1 + \sharp B}.\] 

Going from $\sigma$ to $\overline{\overline{\sigma}^{\sigma^{-1}(N)}}^N$, we take out both these dots, and the signature of the resulting permutation is
\[ \sgn(\overline{\overline{\sigma}^{\sigma^{-1}(N)}}^N) = \sgn(\sigma)(-1)^{\sharp B +1}.\]

Simplifying the signs in \eqref{eq_LC_M}, we get the desired relation \eqref{eq_moins_N}.

Finally, in the case where it is not the two columns but the two rows $\sigma^{-1}(N)$ and $N$ of $\Gamma [\id,\sigma]$ which are identical, one can apply the same procedure of cyclically permuting the rows of the matrix $\tilde{M}$. As a result,
\begin{align*}
\det(\tilde{M}_N^N) &= \sgn(\overline{\sigma}^{\sigma^{-1}(N)})Z(\m''),\\
\det(\tilde{M}_{\sigma^{-1}(N)}^{\sigma(N)}) &=  \sgn(\overline{\sigma}^{N})Z(\m').
\end{align*} 
But a symmetric reasoning on the signatures, we also get the desired relation, which concludes the proof of \eqref{eq_moins_N}.

	\subsection{Proof of irreducibility}\label{proof_irred}
	
\subsubsection{Irreductibility of $Z(\m)\times Z(\m\setminus \Delta_N,\Delta_{\sigma(N)})$:}


As in Remark~\ref{rem_avoid_real}, let us prove this result by induction on $N\geq 3$, the number of segments in $\m$. For $N=3$, assuming $\sigma(3)\neq 3$, then $Z(\m\setminus \Delta_3,\Delta_{\sigma(3)}) = Z(\Delta)$, where $\Delta$ is necessarily a good segment of $\m$ by Proposition~\ref{prop_existence}. By definition, $Z(\m)\times Z(\Delta)$ is irreducible.

Let $N\geq 4$, from Proposition~\ref{prop_existence}, as $\sigma$ avoids the patterns $4231$ and $3412$, we know that either $\m$ is a ladder or it has at least one good segment $\Delta$ which is different from $\Delta_N$ and $\Delta_{\sigma(N)}$. The ladder case has been considered in Section~\ref{sect_ladders}. Otherwise, using Lemma~\ref{lem_good_sub}, we know that $\Delta$ is also a good segment of $\m':=\m\setminus \Delta_N,\Delta_{\sigma(N)}$ (on the same side). 

We can assume without loss of generality that $\Delta$ is a good left segment of $\m$ and $\m'$. Then, as in the proof of Corollary~\ref{cor_2good},
\begin{align*}
Z(\m)\times Z(\m') &\hookrightarrow \underbrace{Z(\m)\times Z(\Delta)}_{\text{irreducible}}\times Z(\m'\setminus \Delta) \cong Z(\Delta) \times Z(\m)\times Z(\m'\setminus \Delta)\\
&\hookrightarrow \underbrace{Z(\Delta) \times Z(\Delta)}_{\text{irreducible}} \times Z(\m\setminus \Delta)\times Z(\m'\setminus \Delta).
\end{align*}
By induction, $Z(\m\setminus \Delta)\times Z(\m'\setminus \Delta)$ is irreducible.

Similarly, 
\[
Z(\m)\times Z(\m') \twoheadleftarrow Z(\m\setminus \Delta)\times Z(\m'\setminus \Delta)\times Z(\Delta) \times Z(\Delta).
\]
We conclude that $Z(\m)\times Z(\m')$ is irreducible by Lemma~\ref{lem_irr}.

\subsubsection{Irreducibility of $Z(\m')\times Z(\m'')$:}

As before, we prove this by induction, this time on $N-r$, where $r=|J|$. If $r=N$, then $\m$ is a ladder, and the result was proven in \cite[Exemple 4.5]{BLM}. If $N>r$, then $\m$ is not a ladder. In that case, either $\Delta_1$ or $\Delta_{\sigma(1)}$ is a good segment of $\m$ and does not form a ladder with $\Delta_N$ and $\Delta_{\sigma(N)}$. This good segment is not one of the $\Delta_{i_k}$ and thus it is a segment of $\m'$ and $\m''$. Let us prove it is a common good segment of $\m'$ and $\m''$.

Let us assume that $\sigma(N)\neq 1$ and that $\Delta_1$ is a good left segment of $\m$. Clearly, $\Delta_1$ does not precede any segment of $\m'$ or $\m''$.

Suppose $\Delta_1$ does not form a ladder with the segments which precedes it in $\m'$. Thus there exists $\Delta,\Delta'$ such that $\Delta \prec \Delta_1$, $\Delta' \prec \Delta_1$ and $\Delta' \subsetneq \Delta$. As $\Delta_1$ is a good segment of $\m$, necessarily exactly one of $\Delta,\Delta'$ is in $\m$ while the other has be shifted. If
$\Delta'\in \m$, then $\Delta=[a_{i_{k}};b_{i_{k+1}}]$ for some $k$. In that case, $\Delta_{i_{k+1}} \prec \Delta_1$ and $\Delta' \nprec \Delta_{i_{k+1}}$, which contradicts the fact that $\Delta_1$ is a good segment of $\m$.

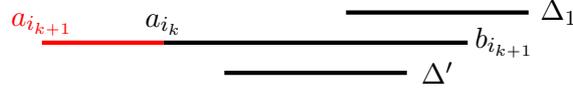
\begin{center}
\begin{tikzpicture}[scale=0.8]
\draw[ultra thick] (5,1) to (8,1);
\draw[ultra thick] (2,0.5) to (7,0.5);
\draw[ultra thick, red] (2,0.5) to (0,0.5);
\draw[ultra thick] (3,0) to (6,0);
\node at (8.5,1) {$\Delta_1$};
\node at (7.6,0.5) {$b_{i_{k+1}}$};
\node at (6.5,0) {$\Delta'$};
\node at (2,0.8) {$a_{i_{k}}$};
\node[red] at (0,0.8) {$a_{i_{k+1}}$};
\end{tikzpicture}
\end{center}
 
If $\Delta=[a_i;b_i]\in \m$, then there is $k$ such that $\Delta'=[a_{i_{k}};b_{i_{k+1}}]$. If both $b_{i_k}>b_i$ and $a_{i_{k+1}}<a_i$ then $i\in I$, which contradicts the fact that $\Delta$ has not been shifted. 
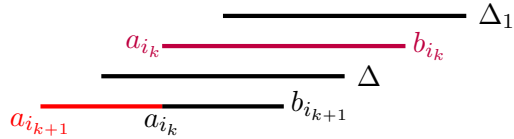
\begin{center}
\begin{tikzpicture}[scale=0.8]
\draw[ultra thick] (3,1.5) to (7,1.5);
\draw[ultra thick, purple] (2,1) to (6,1);
\draw[ultra thick] (1,0.5) to (5,0.5);
\draw[ultra thick] (2,0) to (4,0);
\draw[ultra thick, red] (2,0) to (0,0);
\node at (7.5,1.5) {$\Delta_1$};
\node at (5.4,0.5) {$\Delta$};
\node at (4.6,0) {$b_{i_{k+1}}$};
\node at (2,-0.3) {$a_{i_{k}}$};
\node[red] at (0,-0.3) {$a_{i_{k+1}}$};
\node[purple] at (6.4,1) {$b_{i_{k}}$};
\node[purple] at (1.7,1) {$a_{i_{k}}$};
\end{tikzpicture}
\end{center}
If $b_{i_k}<b_i$, then $\Delta_{i_{k}}, \Delta, \Delta_1$ do not form a ladder in $\m$, if $a_{i_{k+1}}>a_i$ then  $\Delta, \Delta_{i_{k+1}}, \Delta_1$ do not form a ladder in $\m$. In both cases, it contradicts the fact that $\Delta_1$ is a good segment of $\m$.

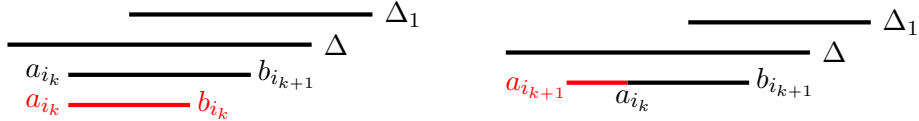
\begin{center}
\begin{tikzpicture}
\node at (0,0) {
\begin{tikzpicture}[scale=0.8]
\draw[ultra thick] (2,1.5) to (6,1.5);
\draw[ultra thick] (0,1) to (5,1);
\draw[ultra thick] (1,0.5) to (4,0.5);
\draw[ultra thick, red] (1,0) to (3,0);
\node at (6.5,1.5) {$\Delta_1$};
\node at (5.4,1) {$\Delta$};
\node at (4.6,0.5) {$b_{i_{k+1}}$};
\node at (0.6,0.5) {$a_{i_{k}}$};
\node[red] at (3.4,0) {$b_{i_{k}}$};
\node[red] at (0.6,0) {$a_{i_{k}}$};
\end{tikzpicture}};
\node at (6.5,0) {
\begin{tikzpicture}[scale=0.8]
\draw[ultra thick] (3,1) to (6,1);
\draw[ultra thick] (0,0.5) to (5,0.5);
\draw[ultra thick] (2,0) to (4,0);
\draw[ultra thick, red] (1,0) to (2,0);
\node at (6.5,1) {$\Delta_1$};
\node at (5.4,0.5) {$\Delta$};
\node at (4.6,0) {$b_{i_{k+1}}$};
\node at (2.1,-0.3) {$a_{i_{k}}$};
\node[red] at (0.5,-0.1) {$a_{i_{k+1}}$};
\end{tikzpicture}};
\end{tikzpicture}
\end{center}
Hence by the criteria in section~\ref{sect_comb}, $\Delta_1$ is good segment of $\m'$. We show in a similar way that $\Delta_1$ is good segment of $\m''$. 

Then, 
\begin{align*}
Z(\m')\times Z(\m'') &\hookrightarrow \underbrace{Z(\Delta_1) \times Z(\Delta_1)}_{\text{irreducible}} \times Z(\m'\setminus \Delta_1)\times Z(\m''\setminus \Delta_1).
\end{align*} 
By induction, $Z(\m'\setminus \Delta_1)\times Z(\m''\setminus \Delta_1)$ is irreducible.

Similarly, 
\[
Z(\m')\times Z(\m'') \twoheadleftarrow Z(\m'\setminus \Delta_1)\times Z(\m''\setminus \Delta_1)\times Z(\Delta_1) \times Z(\Delta_1).
\]
We conclude that $Z(\m')\times Z(\m'')$ is irreducible by Lemma~\ref{lem_irr}.

\section{Relation to quantum affine algebras representations}\label{sect_QAA}
\subsection{Translation of results}
As mentioned above, the result of Theorem~\ref{theo_main} has a quantum affine analog through \emph{quantum affine Schur-Weyl duality}. Indeed, when $q$ is not a root of unity, Chari-Pressley \cite{CP96} have established an equivalence of categories between the category of finite-dimensional representations of the affine Hecke algebra $\dot{H}_{q^2}(n)$ and the category of (level $n$) finite-dimensional representations of the quantum affine algebra $U_q(\widehat{\mathfrak{sl}}_k)$, when $k \geq n$. Moreover, through \emph{type theory} (see for example \cite{Hei}), it is known that finite-dimensional representations of the affine Hecke algebra $\dot{H}_{q^2}(n)$ are equivalent to finite length representations of $\GL_n(F)$.

This equivalence is \emph{monoidal}, in the sense that the parabolic induction of two representations in $\CC$ is translated into the tensor product of the corresponding $U_q(\widehat{\mathfrak{sl}}_k)$-modules.

Instead of multisegments, finite-dimensional irreducible $U_q(\widehat{\mathfrak{sl}}_k)$-modules have been classified \cite{CP95} using \emph{Drinfeld polynomials}, which correspond to their highest-weights. By a process similar to the reduction to cuspidal lines described in the beginning of Section~\ref{sect_rep}, the study of the category of finite-dimensional $U_q(\widehat{\mathfrak{sl}}_k)$-modules amounts to the study of a skeleton Serre subcategory $\mathscr{C}$, introduced by Hernandez-Leclerc (see \cite[Section 3.7]{HL10}), in relation to cluster algebras. Let $\mathscr{R}$ denote the Grothendieck ring of the monoidal category $\mathscr{C}$.

Simple objects in the category $\mathscr{C}$ are then parametrized, up to isomorphism, by monomials in the formal variables $Y_{i,p}$, $(i,p) \in \hat{I}:=\left\lbrace \{1,\ldots, k-1\}\times \BZ \mid i+p+1 \in 2\BZ\right\rbrace$. The correspondence between segments and formal variables is as follows:
\begin{equation}\label{corres_seg_Y}
\begin{array}{ccc}
 {[}a; b] & \mapsto & Y_{b-a+1, -a-b},\\
{[}\frac{1-i-p}{2}; \frac{i-p-1}{2} ] & \mapsfrom &  Y_{i,p}.
\end{array}
\end{equation}

Since we are using the Zelevinsky classification for the representations of $\GL_n(F)$, from now on irreducible $U_q(\widehat{\mathfrak{sl}}_k)$-modules will be denoted $L(M)$, with $M$ their highest loop-weight, in the set of dominant loop-weights:
\begin{equation*}
\hat{P}_\ell := \left\lbrace \prod_{j=1}^N Y_{i_j,p_j} \mid ~ \forall ~ 1\leq j\leq N , (i_j,p_j) \in \hat{I} \right\rbrace.
\end{equation*}

Through this correspondence, ladder representations are usually called \emph{snake modules} in the context of quantum affine algebras. For completeness, recall the definition of snakes modules by Mukhin-Young. For $M=\prod_{j=1}^N Y_{i_j,p_j}\in \hat{P}_\ell$, the simple module $L(M)$ is a snake module if and only if, for all $1\leq j\leq N$, 
\[ p_{j+1} - p_{j} \geq | i_{j+1} - i_j| +2.
\]
It clearly translates to the definition of ladders, as in Definition~\ref{def_ladder}. Note that a definition of snake modules for type $B$ quantum affine algebras was also introduced by Mukhin-Young.

Moreover, as stated above, Theorem~\ref{theo_ladder} from \cite{LM14} was previously established by Mukhin-Young in terms of snake modules.

For $M=\prod_{j=1}^N Y_{i_j,p_j}\in \hat{P}_\ell$ such that $L(M)$ is a snake module, we have the following relation, in the Grothendieck ring $\mathscr{R}$ \cite[Theorem 4.1]{MY12}:
\begin{multline}\label{eq_T_sys_snakes}
\left[L\left(\prod_{j=1}^{N-1} Y_{i_j,p_j}\right)\right]\cdot\left[L\left(\prod_{j=2}^{N} Y_{i_j,p_j}\right)\right] = \left[L\left(\prod_{j=2}^{N-1} Y_{i_j,p_j}\right)\right]\cdot \left[L(M)\right] \\+ \left[L(M')\right]\cdot \left[L(M'')\right],
\end{multline}
where $M',M''$ are called the \emph{neighboring snakes} of $M$, and correspond to $\m'$ and $\m''$ in this case. Note that relation \eqref{eq_T_sys_snakes} was established in \cite{MY12} also in type $B$. Moreover, as in our result, both terms on the left hand side of \eqref{eq_T_sys_snakes} correspond to irreducible modules.

For these reasons, our theorem~\ref{theo_main} is a generalization of \cite[Theorem 4.1]{MY12}, and we have established some new relations between irreducible representations of $U_q(\widehat{\mathfrak{sl}}_k)$.

\begin{ex}\label{ex_QAA}
Let us translate the relations obtained in Example~\ref{ex_segs}:
\begin{enumerate}
	\item For $k\geq 4$, let $M=Y_{2,-5}Y_{3,-2}Y_{1,-2}\in \hat{P}_\ell$. As before, the corresponding regular  representation  $L(M)$ is real. Applying Theorem~\ref{theo_main} gives the following relation:
	\[ L(Y_{2,-5}Y_{3,-2})\cdot L(Y_{3,-2}Y_{1,-2}) = L(M)\cdot L(Y_{3,-2}) + L(Y_{3,-2})\times L(Y_{3,-4}Y_{3,-2}).
	\]

	\item For $k\geq 7$, let $M = Y_{6,-7}Y_{3,-8}Y_{5,-4}Y_{2,-5} \in \hat{P}_\ell$. The corresponding regular  representation  $L(M)$ is real and applying Theorem~\ref{theo_main} gives the following relation:
\begin{multline*}
L(Y_{6,-7}Y_{3,-8}Y_{5,-4})\cdot L(Y_{6,-7}Y_{5,-4}Y_{2,-5}) = L(M)\cdot L(Y_{6,-7}Y_{5,-4}) \\
+ L(Y_{6,-7}Y_{5,-4}Y_{1,-6})\cdot L(Y_{6,-7}Y_{5,-4}Y_{4,-7}).
\end{multline*}
Whereas applying Corollary~\ref{cor_main} gives the following relation:
\begin{multline*}
L(Y_{3,-8}Y_{5,-4}Y_{2,-5})\cdot L(Y_{6,-7}Y_{3,-8}Y_{2,-5}) = L(M)\cdot L(Y_{3,-8}Y_{2,-5}) \\
+ L(Y_{3,-8}Y_{2,-5}Y_{4,-5})\cdot L(Y_{3,-8}Y_{2,-5}Y_{7,-6}).
\end{multline*}
Note that when $k=7$, the right hand side of the last relation simplifies as $L(Y_{3,-8}Y_{2,-5}Y_{4,-5})\cdot L(Y_{3,-8}Y_{2,-5})$.
\end{enumerate}
\end{ex}

\subsection{Relation to cluster algebras}\label{sect_cluster}
In \cite{ACAA}, Hernandez and Leclerc proved that the Grothendieck ring $\mathscr{R}$ had a \emph{cluster algebra} structure for which the initial cluster variables are \emph{Kirillov-Reshetikhin} modules (or Speh representations, as in Definition~\ref{def_speh}). Moreover, one of the key ingredients used for this result is the fact that the $T$-system relations (of which the Mukhin-Young extended $T$-systems are generalizations) correspond to \emph{exchange relations} in the cluster algebra structure. The same authors also conjectured \cite[Conjecture 5.2]{ACAA} that the cluster variables were in bijection with the prime real simple modules. Part of this conjecture was proven by Kashiwara-Kim-Oh-Park in \cite{KKOP21}, where they proved that all cluster variables correspond to prime real simple modules.

In \cite{DLL20} Duan-Li-Luo proved that prime snake modules correspond to cluster variables, thus proving Hernandez-Leclerc's conjecture for snake modules, and for that purpose introduced new relations in the Grothendieck ring $\mathscr{R}$, which they interpreted as exchange relations. However, it is unclear whether (some of) the Mukhin-Young extended $T$-systems can be interpreted as exchange relations.

One of the motivations behind this work was to obtain more generalizations of the $T$-system relations, which could be interpreted as exchange relations. We conjecture that, equipped with more explicit relations such as \eqref{eq_moins_N} and \eqref{eq_moins_1}, one could prove that all prime real regular representations (for which there exists the criterion of Theorem~\ref{theo_pat} \cite{LM18}) correspond to cluster variables. 

However, we already observe that not all relations \eqref{eq_moins_N} and \eqref{eq_moins_1} have the form of an exchange relation. For example, in the relation in Example~\ref{ex_QAA} (1), one of the factors in the left hand side is not prime $L(Y_{3,-2}Y_{1,-2}) \cong  L(Y_{3,-2})\cdot L(Y_{1,-2})$. Thus the left hand side is a product of three prime irreducible modules, and the relation cannot be an exchange relation.


\begin{appendix}

	\section{Ferrers boards}\label{sect_Ferrers}
	
Permutations in $S_N$ can be represented by placing dots in an $N\times N$-grid. For all $1\leq i \leq N$, place a dot in the box $(i,\sigma(i))$. Then the set $\Gamma[\id, \sigma]\subset [N]^2$ can be represented in the grid by colouring the boxes $(i,\sigma'(i))$, for $\sigma'\leq \sigma$.

\begin{ex}
The grid corresponding to the permutation $\sigma = 152463$ is
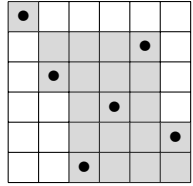
\begin{center}
\begin{tikzpicture}
\draw[step=0.4cm,color=black] (0,0) grid (2.4,2.4);
\fill[color=gray, opacity=0.3] (0,2) rectangle (0.4,2.4);
\fill[color=gray, opacity=0.3] (0.4,1.2) rectangle (2,2);
\fill[color=gray, opacity=0.3] (0.8,0.8) rectangle (2,1.2);
\fill[color=gray, opacity=0.3] (0.8,0) rectangle (2.4,0.8);
\node at (0.2,2.2) {$\bullet$};
\node at (1.8,1.8) {$\bullet$};
\node at (0.6,1.4) {$\bullet$};
\node at (1.4,1) {$\bullet$};
\node at (2.2,0.6) {$\bullet$};
\node at (1,0.2) {$\bullet$};
\end{tikzpicture}
\end{center}
\end{ex}

\begin{rem}
By the study of Sj\"{o}strand \cite{S07}, the set $\Gamma[\id, \sigma]$ is a \emph{right-aligned Skew Ferrers board}, in particular it is the union of the rectangles \begin{tikzpicture}[scale=0.8]
\draw[step=0.4cm,color=black] (0,0) grid (1.2,0.8);
\fill[color=gray, opacity=0.3] (0,0) rectangle (1.2,0.8);
\node at (1,0.6) {$\bullet$};
\node at (0.2,0.2) {$\bullet$};
\end{tikzpicture} for all pairs of (not necessarily distinct) dots $\bullet$.
\end{rem}

\begin{lem}\label{lem_ferrers}
The $\sigma$ be a permutation in $S_N$ which avoids the pattern $3412$, then either the columns $\sigma(N)$ and $N$ or the rows $\sigma^{-1}(N)$ and $N$ of $\Gamma[\id, \sigma]$ are identical.
\end{lem}

\begin{proof}
If the columns $\sigma(N)$ and $N$ of $\Gamma[\id, \sigma]$ are different, then there is a dot above $(\sigma^{-1}(N),N)$ and to the right of $(N,\sigma(N))$. 
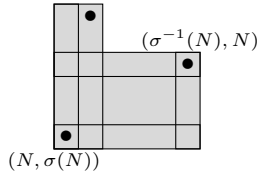
\begin{center}
\begin{tikzpicture}[scale=0.8]
\fill[color=black,fill=gray, opacity=0.3] (0,0) rectangle (2.4,1.6);
\fill[color=black,fill=gray, opacity=0.3] (0,1.6) rectangle (0.8,2.4);
\draw (0,0) rectangle (0.4,2.4);
\draw (0,1.2) rectangle (2.4,1.6);
\draw (2,0) rectangle (2.4,1.6);
\draw (0,0) rectangle (2.4,0.4);
\node at (2.2,1.4) {$\bullet$};
\node at (0.2,0.2) {$\bullet$};
\node at (0.6,2.2) {$\bullet$};
\draw (0,0) rectangle (0.8,2.4);
\node at (2.4,1.8) {\tiny $(\sigma^{-1}(N),N)$};
\node at (0,-0.2) {\tiny $(N,\sigma(N))$};
\end{tikzpicture}
\end{center}
Similarly, if the rows $\sigma^{-1}(N)$ and $N$ are different, then there is a dot to the left of $(N,\sigma(N))$ and below $(\sigma^{-1}(N),N)$.

Thus if both the columns $\sigma(N)$ and $N$ and the rows $\sigma^{-1}(N)$ and $N$ are different, then there a $3412$ configuration, which is a contradiction.
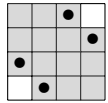
\begin{center}
\begin{tikzpicture}[scale=0.8]
\draw[step=0.4cm,color=black] (0,0) grid (1.6,1.6);
\fill[color=gray, opacity=0.3] (0,0.4) rectangle (1.2,1.6);
\fill[color=gray, opacity=0.3] (0.4,0) rectangle (1.6,0.4);
\fill[color=gray, opacity=0.3] (1.2,0.4) rectangle (1.6,1.2);
\node at (1,1.4) {$\bullet$};
\node at (1.4,1) {$\bullet$};
\node at (0.2,0.6) {$\bullet$};
\node at (0.6,0.2) {$\bullet$};
\end{tikzpicture}
\end{center}
\end{proof}

The following is clear from the definition of the signature.
\begin{lem}\label{lem_sgn}
The signature of the permutation $\sigma$ is equal to $(-1)^\ell$, where $\ell$ is the sum over all dots $\bullet$ of the number of dots strictly above and to the right of $\bullet$.
\end{lem}
\end{appendix}

\bibliographystyle{alpha}
\bibliography{article}

\end{document}